\documentclass[12pt]{amsart}

\usepackage[utf8]{inputenc}
\usepackage[T1]{fontenc}

\usepackage[english]{babel}

\usepackage{amsmath}
\usepackage{amssymb}
\usepackage{amsthm}

\usepackage{dsfont}
\usepackage{mathrsfs}
\usepackage{stmaryrd}

\usepackage{colonequals}

\usepackage{xcolor}

\usepackage{graphicx}

\usepackage{caption}

\usepackage{tikz}
\usetikzlibrary{arrows}
\usepackage{tikz-cd}
\usetikzlibrary{fpu}

\usepackage{pgf}
\usepackage{pgfplots}
\pgfplotsset{compat=1.15}

\usepackage{forest}

\usepackage{changepage}

\usepackage{fancyhdr}
\usepackage[backend=biber,style=numeric,maxbibnames = 99]{biblatex}
\usepackage{todonotes}

\usepackage{hyperref}
\hypersetup{
 colorlinks,
 linkcolor={red!50!black},
 citecolor={blue!50!black},
 urlcolor={blue!80!black}
}

\theoremstyle{plain}
\newtheorem{teo}{Theorem}[section]
\newtheorem{prop}[teo]{Proposition}
\newtheorem{cor}[teo]{Corollary}
\newtheorem{lem}[teo]{Lemma}

\theoremstyle{definition}
\newtheorem{defi}[teo]{Definition}

\theoremstyle{remark}
\newtheorem{rema}[teo]{Remark}

\newtheorem{exemple}[teo]{Example}

\newtheorem{notat}[teo]{Notation}

\DeclareMathOperator{\Hom}{Hom}

\DeclareMathOperator{\Op}{Op}
\DeclareMathOperator{\End}{End}
\DeclareMathOperator{\SymMon}{SymMon}

\DeclareMathOperator{\emb}{emb}

\DeclareMathOperator{\Moore}{Moore}
\DeclareMathOperator{\Conf}{Conf}
\DeclareMathOperator{\UConf}{UConf}
\DeclareMathOperator{\id}{id}

\DeclareMathOperator{\Top}{Top}

\DeclareMathOperator{\Pos}{Pos}

\DeclareMathOperator{\HT}{HT}
\DeclareMathOperator{\Ob}{Ob}
\DeclareMathOperator{\Mor}{Mor}

\DeclareMathOperator{\im}{im}

\DeclareMathOperator{\Aut}{Aut}
\DeclareMathOperator{\dist}{dist}
\DeclareMathOperator{\aaaaah}{max}
\DeclareMathOperator{\pc}{pc}

\newcommand{\iso}{\overset{\sim}{=}}
\newcommand{\op}{\mathcal{O}}
\newcommand{\C}{\mathcal{C}}
\newcommand{\s}{\mathcal{S}}
\newcommand{\E}{\mathcal{E}}
\newcommand{\F}{\mathcal{F}}
\newcommand{\V}{\mathcal{V}}

\newcommand{\G}{\mathcal{G}}

\title{Stable homology of Higman--Thompson groups via scanning methods}
\author{Marie-Camille Delarue}
\thanks{Université Paris Cité and Sorbonne Université, CNRS, IMJ-PRG, F-75013 Paris, France.}

\begin{document}

\begin{abstract}
    The Higman--Thompson groups $V_{n,r}$ consist of piecewise linear automorphisms of $r$ intervals where cut points and slopes are $n$-adic.
    Szymik and Wahl prove homological stability for this family of groups as $r$ increases, and compute the stable homology to be that of the infinite loop space of the Moore spectrum.
    We give a new proof of this result using scanning methods on a topological model for the disjoint union of these groups.
    We use Thumann's framework of operad groups to build this model.
\end{abstract}

\maketitle
\tableofcontents

\section{Introduction}

The Thompson groups were introduced by Thompson in the 1960s in the context of word problems, and as potential candidates to disprove the von Neumann conjecture \cite{thompson}. 
There are three original Thompson groups, $F$, $T$, and $V$, and they arise as subgroups of piecewise linear bijections of an interval where all points of non-differentiability are dyadic. 
It is useful to think of them as groups of paired diagrams of forests as in \cite{skipperwu}.
For instance, a piecewise linear map between two intervals is given by two ways of cutting an interval into a certain number of pieces and then a permutation which matches the pieces together.
This can be represented by tree diagrams (see figure \ref{illustrationgroup} and the formal definition in section \ref{trees}).
Higman later introduced generalizations of these families of groups \cite{higman}, now called the Higman--Thompson groups. 
Our particular objects of study, the groups $V_{n,r}$, are subgroups of piecewise linear bijections of a disjoint union of $r$ intervals, where the cut points are $n$-adic numbers.

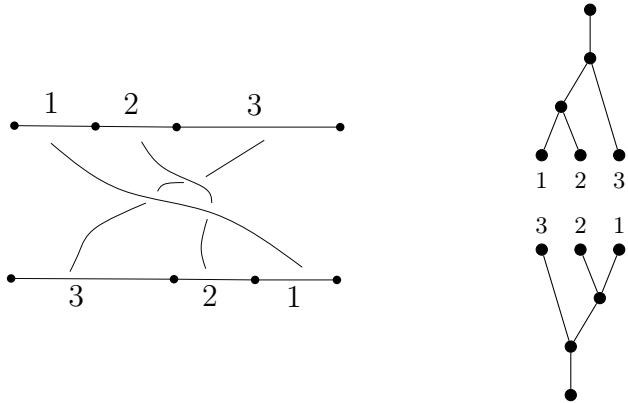
\begin{figure}
    \centering
    \begin{minipage}{0.4\linewidth}
        \def\svgwidth{\hsize}
        \begingroup%
  \makeatletter%
  \providecommand\color[2][]{%
    \errmessage{(Inkscape) Color is used for the text in Inkscape, but the package 'color.sty' is not loaded}%
    \renewcommand\color[2][]{}%
  }%
  \providecommand\transparent[1]{%
    \errmessage{(Inkscape) Transparency is used (non-zero) for the text in Inkscape, but the package 'transparent.sty' is not loaded}%
    \renewcommand\transparent[1]{}%
  }%
  \providecommand\rotatebox[2]{#2}%
  \newcommand*\fsize{\dimexpr\f@size pt\relax}%
  \newcommand*\lineheight[1]{\fontsize{\fsize}{#1\fsize}\selectfont}%
  \ifx\svgwidth\undefined%
    \setlength{\unitlength}{196.50713102bp}%
    \ifx\svgscale\undefined%
      \relax%
    \else%
      \setlength{\unitlength}{\unitlength * \real{\svgscale}}%
    \fi%
  \else%
    \setlength{\unitlength}{\svgwidth}%
  \fi%
  \global\let\svgwidth\undefined%
  \global\let\svgscale\undefined%
  \makeatother%
  \begin{picture}(1,0.47636163)%
    \lineheight{1}%
    \setlength\tabcolsep{0pt}%
    \put(0,0){\includegraphics[width=\unitlength,page=1]{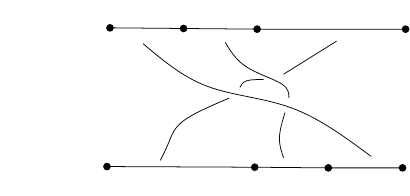}}%
    \put(0.33313195,0.43730477){\color[rgb]{0,0,0}\makebox(0,0)[lt]{\lineheight{1.25}\smash{\begin{tabular}[t]{l}$1$\end{tabular}}}}%
    \put(0.51003654,0.43500732){\color[rgb]{0,0,0}\makebox(0,0)[lt]{\lineheight{1.25}\smash{\begin{tabular}[t]{l}$2$\end{tabular}}}}%
    \put(0.78343444,0.43270986){\color[rgb]{0,0,0}\makebox(0,0)[lt]{\lineheight{1.25}\smash{\begin{tabular}[t]{l}$3$\end{tabular}}}}%
    \put(0.38827098,0.00767948){\color[rgb]{0,0,0}\makebox(0,0)[lt]{\lineheight{1.25}\smash{\begin{tabular}[t]{l}$3$\end{tabular}}}}%
    \put(0.6846435,0.00767954){\color[rgb]{0,0,0}\makebox(0,0)[lt]{\lineheight{1.25}\smash{\begin{tabular}[t]{l}$2$\end{tabular}}}}%
    \put(0.87073781,0.00538208){\color[rgb]{0,0,0}\makebox(0,0)[lt]{\lineheight{1.25}\smash{\begin{tabular}[t]{l}$1$\end{tabular}}}}%
  \end{picture}%
\endgroup%

    \end{minipage}
    \begin{minipage}{0.4\linewidth}
        \begin{center}
        \begin{forest}
            for tree = {s sep =  10pt, l-=5mm, inner sep = 1.5pt, if n children=0{tier=terminal}{}}
            [, circle, fill = black, draw,
            [, circle, fill = black, draw,
              [, circle, fill = black, draw,
              [, circle, label = {[font=\scriptsize]south:1}, fill = black, draw][, circle, label = {[font=\scriptsize]south:2}, fill = black, draw]
              ][, circle, label = {[font=\scriptsize]south:3}, fill = black, draw]
            ]]
        \end{forest}
    \end{center}
    \begin{center}
        \begin{forest}
            for tree = {s sep =  10pt, l-=5mm, inner sep = 1.5pt, if n children=0{tier=terminal}{}}
            [, circle, fill = black, draw, grow=north,
            [, circle, fill = black, draw, grow=north,
            [, circle, fill = black, draw,grow=north,
              [, circle, label = {[font=\scriptsize]north:1}, fill = black, draw,grow=north][, circle, label = {[font=\scriptsize]north:2}, fill = black, draw,grow=north]
              ]
              [, circle, label = {[font=\scriptsize]north:3}, fill = black, draw,grow=north]
            ]]
        \end{forest}
    \end{center}
    \end{minipage}
\begin{tikzpicture}
\end{tikzpicture}
\caption{Illustration of an element in Thompson's group $V$}
\label{illustrationgroup}
\end{figure}

\paragraph{\textbf{Homological stability.}} A natural phenomenon to investigate is whether the homology of these groups stabilizes when the number $r$ of intervals is big enough.
Szymik and Wahl prove homological stability for the family of groups $V_{n,r}$ when $r$ goes to infinity and also compute their stable homology \cite{szymikwahl}.
In this particular example, the homology stabilizes instantly.
This result was reproved as an application of a wide framework for scissors congruence groups by Kupers, Lemann, Malkiewich, Miller and Sroka \cite{KLMMS}.
It also appears in work of Tanner \cite{tanner} building on Li's work on stable homology of topological full groups \cite{li}.
In this paper, we give an different proof of theorem \ref{mainthm} which expresses the stable homology of the Higman--Thompson groups, i.e. the homology of the colimit $V_{n,\infty}$, as the homology of a certain infinite loop space.

\begin{teo}[\cite{szymikwahl}]\label{mainthm}
    For all $n\geq 2$, there is a homology equivalence :
\[ BV_{n,\infty} \overset{H_*}{\simeq} \Omega^\infty_0 \mathbb{M}_{n-1}.\]
\end{teo}
Here $\Omega^\infty_0 \mathbb{M}_{n-1}$ is the basepoint component of the infinite loop space underlying the mod $(n-1)$ Moore spectrum, defined as the cofiber of the map $\mathbb{S}\xrightarrow{\times(n-1)}\mathbb{S}$.

\paragraph{\textbf{Scanning techniques.}}
The scanning techniques we use were originally developed to study the stable homology of configuration spaces \cite{mcduff} and spaces of holomorphic functions of closed surfaces \cite{segal79}, 
and then adapted to the homology of mapping class groups of surfaces \cite{madsenweiss} and of higher dimensional manifolds \cite{gmtw,grwI,sorenrwModuliMonoids}.
Galatius \cite{soren} then adapted these methods to moduli spaces of graphs to compute the stable homology of automorphism groups of free groups.

The main idea of scanning is to exhibit a space of embeddings of certain topological objects as a loop space of local images of those objects.
Consider, for instance, a space of embeddings of a tree $T$ in $\mathbb{R}^{N}$, denoted by $\emb(T,\mathbb{R}^{N})$. 
The scanning map is defined on $\mathbb{R}^N\times \emb(T,\mathbb{R}^{N})$ by sending a pair $(x,\phi)$ to what is visible of $\phi$ in a neighborhood of the point $x$, as if there were a magnifying lens at $x$ (see figure \ref{scanningtrees}).
This map factors through the sphere $S^N$. By adjunction, this yields a map from $\emb(T,\mathbb{R}^{N})$ into an $N$-loop space.
\begin{figure}[h!]
    \centering
    \tiny
    \def\svgwidth{\hsize}
    \begingroup%
  \makeatletter%
  \providecommand\color[2][]{%
    \errmessage{(Inkscape) Color is used for the text in Inkscape, but the package 'color.sty' is not loaded}%
    \renewcommand\color[2][]{}%
  }%
  \providecommand\transparent[1]{%
    \errmessage{(Inkscape) Transparency is used (non-zero) for the text in Inkscape, but the package 'transparent.sty' is not loaded}%
    \renewcommand\transparent[1]{}%
  }%
  \providecommand\rotatebox[2]{#2}%
  \newcommand*\fsize{\dimexpr\f@size pt\relax}%
  \newcommand*\lineheight[1]{\fontsize{\fsize}{#1\fsize}\selectfont}%
  \ifx\svgwidth\undefined%
    \setlength{\unitlength}{453.90062227bp}%
    \ifx\svgscale\undefined%
      \relax%
    \else%
      \setlength{\unitlength}{\unitlength * \real{\svgscale}}%
    \fi%
  \else%
    \setlength{\unitlength}{\svgwidth}%
  \fi%
  \global\let\svgwidth\undefined%
  \global\let\svgscale\undefined%
  \makeatother%
  \begin{picture}(1,0.29158413)%
    \lineheight{1}%
    \setlength\tabcolsep{0pt}%
    \put(0,0){\includegraphics[width=\unitlength,page=1]{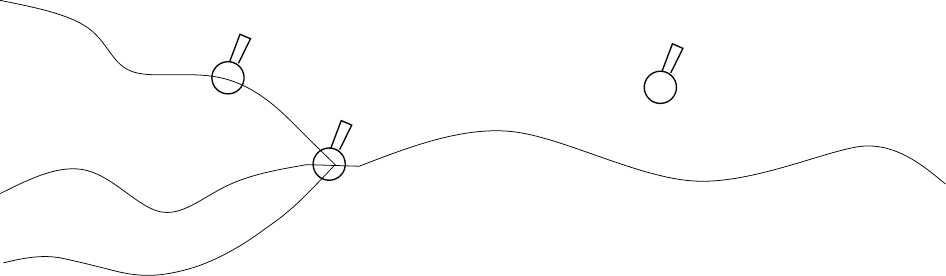}}%
  \end{picture}%
\endgroup%

    \caption{Philosophy of the scanning map}
    \label{scanningtrees}
    \end{figure}
This identification enables us to compute the aforementioned space of embeddings.

\paragraph{\textbf{Proof strategy.}}
In order to put this to use, we rely on work of Thumann \cite{thumann} concerning operads, i.e., objects that encode algebraic structures.
For instance, the little cubes operads encode structures which are only commutative up to certain homotopies.
Thumann establishes a link between the Higman--Thompson groups and the PROPs associated to certain suboperads of the little cubes operads, called cube cutting operads.
We then build a topological model for the groups as a kind of cobordism category where objects are configurations of points and morphisms are embeddings of trees. 
We can then construct a scanning map to study the homotopy type of this category.
This is an elaboration on our previous paper about the stable homology of symmetric groups \cite{delarue0}.

We plan to generalize this approach to groups of bijections of $d$-dimensional cubes for $d>1$ (already studied by Thumann \cite{thumann} and Kupers--Lemann--Malkiewicz--Miller--Sroka \cite{KLMMS}) in future work.
The higher dimensional cube cutting operads are Boardman--Vogt tensor products of the 1-dimensional cube cutting operad, so we will need to build a topological model which encodes this product.
The advantage of the methods developed here is that they provide a general framework for studying groups that can be modeled topologically by embedded combinatorial objects.

\paragraph{\textbf{Organization of the paper.}} In section \ref{exposition}, we provide some background on simplicial and semi-simplicial spaces, as well as other technical tools.
We introduce the Higman--Thompson groups and the connection to operads established by Thumann in section \ref{trees}.
Section \ref{embeddings} sets up the topological model for the Higman--Thompson groups.
Our proof of theorem \ref{mainthm} is detailed in several steps in section \ref{deloopings} where the scanning machinery is set up, 
and section \ref{zoom} which studies the loop space obtained after scanning.

\paragraph{\textbf{Acknowledgements.}}
The author is grateful to Najib Idrissi and Nathalie Wahl for their supervision and encouragement, as well as to S{\o}ren Galatius for initially suggesting the argument.
The author was partially supported by the project ANR-22-CE40-0008 SHoCoS, and by the Danish National Research Foundation through the Copenhagen Center for Geometry and Topology (DNRF151).

\section{Preliminaries}\label{exposition}

Let us here briefly recall the basics of (semi-)simplicial spaces, as well as some useful technical lemmas.
The model for $\Top$ used here is that of compactly generated Hausdorff spaces equipped with the Serre model structure.
Let $\Delta$ be the category with objects the finite sets $[n] = \{0 < \dots < n\}$ and with morphisms from $[n]$ to $[m]$ the order-preserving set maps.
A simplicial space is a functor $\Delta^{op}\rightarrow \Top$.
A cosimplicial space is a functor $\Delta \rightarrow \Top$.
We denote the standard topological $n$-simplex by $\Delta^n_{\Top}$:
\[\Delta^n_{\Top} \colonequals \{x\in\mathbb{R}^{n+1} \mid x_i \geq 0 \text{ for } 0\leq i\leq n, \text{ and } x_0+\dots +x_n = 1.\}\]
These form a cosimplicial space in a natural way.
Let $X_\bullet$ be a simplicial space. 
Its geometric realization $\vert X_\bullet \vert$ is defined as the space $\displaystyle \sqcup_n \left( X_n\times \Delta^n_{\Top}\right)/\sim$ where $(x, f_*p)\sim (f^*x,p)$ for every simplicial map $f:[k]\rightarrow [l]$.

We work with categories that have some extra structure, namely, they are internal to topological spaces.
\begin{defi}\cite[section~2]{segal68}
    A category internal to topological spaces $\C$ (or topological category) is given by a pair of topological spaces, called the object space $\Ob(\C)$ and the morphism space $\Mor(\C)$, 
    together with continuous maps:
    \begin{itemize}
        \item $u:\Ob(\C)\rightarrow \Mor(\C)$ which sends an object to the identity morphism;
        \item $s,t:\Mor(\C)\rightarrow \Ob(\C)$ which send an arrow to its source (resp. target);
        \item $\circ:\Mor(\C)\times_{t,s}\Mor(\C)\rightarrow \Mor(\C)$, which is the composition of morphisms.
    \end{itemize}
    These maps have to satisfy the following conditions:
    \begin{align*}
        (f\circ g)\circ h = f\circ(g\circ h)\ \ \ \ \ &t(f\circ g) = t(f) &s(f\circ g) = s(g)\\
        t\circ u = s\circ u = \id\ \ \ \ \ &f\circ u(t(f)) = f &u(s(f))\circ f = f.
    \end{align*}
\end{defi}
If there is no map $u$, the category is said to be non-unital \cite{ebertbible}.
\begin{defi}\cite[Definition~3.1]{ebertbible}
    A non-unital topological category $\C$ is said to have \emph{weak left units} if for every object $c\in\Ob(\C)$, there exists a morphism $u:c\rightarrow c'$ such that the induced map 
    \[\Hom_\C(-,c)\xrightarrow{u\circ} \Hom_\C(-,c')\]
    is a weak homotopy equivalence.
    The category $\C$ is said to have \emph{weak right units} if for every object $c\in\Ob(\C)$, there exists a morphism $u:c'\rightarrow c$ such that the induced map 
    \[\Hom_\C(c,-)\xrightarrow{\circ u} \Hom_\C(c',-)\]
    is a weak homotopy equivalence.
\end{defi}
\begin{defi}\cite[Definition~3.5]{ebertbible}
    A topological category is \emph{fibrant} if the map $(s,t): \Mor(\C)\rightarrow \Ob(\C)\times\Ob(\C)$ is a fibration.
\end{defi}

Let us now introduce semi-simplicial spaces, which we will often prefer to simplicial spaces.
Let $\Delta_+$ be the subcategory of $\Delta$ which contains all objects but only the injective maps.
A semi-simplicial space is a functor $\Delta_+^{op}\rightarrow \Top$.
There is a forgetful functor from simplicial spaces to semi-simplicial spaces, induced by the inclusion $\Delta_+\hookrightarrow \Delta$. It amounts to forgetting the degeneracies.

There is a notion of geometric realization for semi-simplicial spaces, called the \emph{thick geometric realization}, $\vert\vert X_\bullet \vert\vert$, defined as $\displaystyle \sqcup_n \left(X_n\times \Delta^n_{\Top}\right)/\sim$ where $(x, f_*p)\sim (f^*x,p)$ for every face map $f:[k]\hookrightarrow [l]$ in $\Delta_+$.
It is called ``thick'' because there are no degeneracies in semi-simplicial spaces to collapse degenerate simplices.
When $X_\bullet$ is a simplicial space, the geometric realization is a quotient of the thick geometric realization.

The nerve of a topological category $\C$ is the simplicial space whose $n$-simplices are given by chains of $n$ composable morphisms:
\[(N\C)_n = \Mor(\C)\times_{\Ob(\C)} \dots \times_{\Ob(\C)} \Mor(\C).\]
The classifying space $B\C$ of a category $\C$ is the geometric realization of its nerve.
We denote by $\mathbb{B}\C$ the thick geometric realization of its nerve, seen as a semi-simplicial space.
When a category is non-unital, its classifying space is only a semi-simplicial space.

The usual notion of Segal space \cite{segal68} extends to semi-simplicial spaces.
Recall that for $X_\bullet$ a semi-simplicial space, the $n$th Segal map is the map
\[X_n \rightarrow X_1\times_{X_0}\dots\times_{X_0} X_1\] which comes from the maps $X_n\rightarrow X_1$ induced by the following maps $a^{i}$ in the simplex category:
\begin{align*}
    a^{i}:[1]&\rightarrow [n],\ \ \ \ 0\mapsto i, \ \ \ \ 1\mapsto i+1, \ \ \ \ 0\leq i \leq n-1
\end{align*}
A result of Segal \cite[Proposition~1.5]{segal74} states that for a simplicial Segal space such that $X_1$ is path-connected, there is a weak equivalence $X_1\simeq\Omega\vert\vert X_\bullet\vert\vert$.
This actually extends to semi-simplicial spaces:
\begin{lem}\label{semiSegal}
Let $X_\bullet$ be a semi-simplicial Segal space where $X_1$ is path-connected. Then there is a weak equivalence $X_1\simeq\Omega\vert\vert X_\bullet\vert\vert$.
\end{lem}

We prefer working with semi-simplicial spaces because they offer more flexibility and need less structure.
One of the main results we need in order to work with semi-simplicial spaces is proposition \ref{levelwisereal}.
It implies that if two semi-simplicial spaces are levelwise equivalent, then their thick realizations are equivalent, which is not true in general for simplicial spaces and regular realizations.
In order to state this proposition in detail, we need the following definition:
\begin{defi}
A map $p:E\rightarrow B$ is called a Serre microfibration if, for every $k\geq 0$ and every commutative square
\begin{center}
    \begin{tikzcd}
    D^k\times\{0\}\arrow{d} \arrow{r} &E\arrow{d}{p}\\
    D^k\times[0,1]\arrow{r} &B,
    \end{tikzcd}
\end{center}
there exists $\epsilon>0$ and a lifting $h:D^k\times [0,\epsilon]\rightarrow E$ such that the following diagram commutes:
\begin{center}
\begin{tikzcd}[row sep = 0.4cm, column sep = 0.4cm]
D^k\times\{0\} \arrow[rr] \arrow[dd] \arrow[dr] && E\arrow[dd]\\
&D^k\times[0,\epsilon] \arrow[dl]\arrow[ur,dashrightarrow]\\
D^k\times[0,1]\arrow[rr] &&B.
\end{tikzcd}
\end{center}
\end{defi}

\begin{exemple}
    A Serre fibration is a Serre micro-fibration.
    Let $p:E\rightarrow B$ be a Serre fibration, and $V$ an open subset of $E$. 
    Then the restriction of $p$ to $V$ is a Serre microfibration, but not necessarily a Serre fibration (see for instance Raptis in \cite{raptis}.)
\end{exemple}

\begin{lem}\cite[Lemma~2.2]{weisslemma}
A Serre micro-fibration with weakly contractible fibers is a Serre fibration, and therefore a weak homotopy equivalence.
\end{lem}

\begin{prop}\label{levelwisereal}\cite[Propositions~2.7-2.8]{grwI}.
\begin{itemize}
\item[$\bullet$] Let $f_\bullet:X_\bullet\rightarrow Y_\bullet$ be a map of semi-simplicial spaces and $n\in\mathbb{N}\cup\{\infty\}$ such that for all $p\in\mathbb{N}$, $f_p:X_p\rightarrow Y_p$ is $(n-p)$-connected. Then $||f_\bullet||:||X_\bullet||\rightarrow ||Y_\bullet||$ is $n$-connected.
\item[$\bullet$] Consider a semi-simplicial set $Y$ (viewed as a discrete space) and a Hausdorff space $Z$. Let $X_\bullet \subset Y_\bullet \times Z$ be a sub-semi-simplicial space which is an open subspace in every degree.
Then the map $\pi:||X_\bullet|| \rightarrow Z$ is a Serre microfibration.
\end{itemize}
\end{prop}
A statement analogous to the first point for the regular geometric realization is false unless the simplicial space satisfies certain conditions, mentioned in definition~\ref{proper}.

\begin{defi}\cite[Definition~A.4]{segal74}\label{proper}
A simplicial space $X_\bullet$ is good if the maps $s_i(X_{p-1})\rightarrow X_p$ are closed Hurewicz cofibrations for every $i$ and $p$.
\end{defi}

\begin{lem}\label{realizations}\cite[Proposition~A.1.(iv)]{segal74}
Let $X$ be a simplicial space. If $X_\bullet$ is good, then the quotient map $||X_\bullet||\rightarrow |X_\bullet|$ is a weak equivalence.
\end{lem}

Because many of the simplicial spaces we use are nerves of categories, we need a condition to check directly whether the simplicial and semi-simplicial nerves of a category are equivalent.
We will say that a topological category $\C$ is \emph{well-pointed} if its nerve $N\C$ is a good simplicial space \cite{TopQuillenA}. 

\begin{defi}\label{NDRgood}
For $B$ a topological space, let $\mathrm{CGH}/B$ be the category of compactly generated Hausdorff spaces over $B$. 
A neighborhood-deformation-retract pair (or NDR pair) over $B$ is a pair $(X,A)$ in $\mathrm{CGH}/B$, with $A\subset X$, such that there exist maps $u:X\rightarrow I\times B$ and $h: X\times I\rightarrow X$
satisfying the following conditions:
\begin{enumerate}
    \item $A = u^{-1}(\{0\}\times B)$,
    \item $h(-,0) = \id_X$,  $h\vert_{A\times I} = pr_A$,
    \item $h(x,1)\in A$ for $x\in u^{-1}([0,1)\times B)$,
\end{enumerate}
These are maps of $\mathrm{CGH}/B$ so they have to make the following diagrams commute:
\begin{center}
\begin{minipage}{0.4\linewidth}
\begin{center}
    \begin{tikzcd}
    X\arrow[dr]\arrow[rr,"u"] && I\times B \arrow[dl]\\
    &B
    \end{tikzcd}
\end{center}
\end{minipage}
\begin{minipage}{0.4\linewidth}
\begin{center}
    \begin{tikzcd}
    X\times I\arrow[dr]\arrow[rr,"h"] && X \arrow[dl]\\
    &B.
    \end{tikzcd}
\end{center}
\end{minipage}
\end{center}
Note that the map $X\times I\rightarrow B$ does not need to be the projection to $X$ composed with the map $X\rightarrow B$.
\end{defi}

\begin{prop}\cite[Proposition~10]{TopQuillenA}
Let $\C$ be a topological category. \\
If $(\Mor(\C), \Ob(\C))$ is an NDR-pair over $\Ob(\C)\times \Ob(\C)$, then $N\C$ is a good simplicial space.
\end{prop}

\subsection{Object change in classifying spaces of categories}
Let us introduce a technical result about classifying spaces of categories proven by Ebert and Randal-Williams \cite{ebertbible}.
It states that we can sometimes change the homotopy type of the object space of a topological category without changing the homotopy type of its classifying space.
\begin{defi}\cite[section~5]{ebertbible}\label{basechangeobjects}
Let $\C$ be a (not necessarily unital) topological category, and $f$ a continuous map $X\rightarrow \Ob(\C)$. We define the category $f^*\C$ as follows.
The space of objects is $X$. The map $f$ is therefore a map $F_0:\Ob(f^*\C)\rightarrow \Ob(\C)$. The space of morphisms of $f^*\C$ is the pullback
\begin{center}
\begin{tikzcd}
    \Mor(f^*\C) \arrow{r}{F_1}\arrow{d} &\Mor(\C)\arrow{d}{(s,t)}\\
    \Ob(f^*\C)\times\Ob(f^*\C)\arrow{r}{F_0\times F_0} &\Ob(\C)\times\Ob(\C).
\end{tikzcd}
\end{center}
The maps $s,t:\Mor(f^*\C)\rightarrow \Ob(f^*\C)$ are given by the left hand part of the diagram. 
Composition follows from the universal property of the pullback.
Combining $F_0$ and $F_1$ yields a continuous functor $F:f^*\C\rightarrow \C$.
\end{defi}

\begin{prop}\label{basechangeproperty}\cite[Theorem~5.2]{ebertbible}
    If $\C$ is a non-unital topological category which is fibrant and has weak right (or left) units, and the map $f:X\rightarrow \Ob(\C)$ is 0-connected, then $\mathbb{B}F: \mathbb{B}f^*\C\rightarrow \mathbb{B}\C$ is a weak equivalence.
\end{prop}

We consider topological categories with extra structure:
\begin{defi}\cite[Chapter 4]{may}
An $E_n$-algebra in a symmetric monoidal category $\C$ is an object $X$ of $\C$ equipped with maps $E_n(k)\otimes X^{\otimes k}\rightarrow X$.
The model we use for $E_n$ is that of the little $n$-cubes operad.
\end{defi}

\section{Description of the Higman--Thompson groups}\label{trees}

Let us now introduce Higman--Thompson groups, defined by Higman \cite{higman} as generalizations of the original Thompson groups \cite{thompson}.
They can be seen as automorphisms of Cantor algebras.
A Cantor algebra of type $n$ is a set $X$ with a bijection $X^n\cong X$.
For $n\geq 2$ and $r\geq 1$, the Higman--Thompson group $V_{n,r}$ is the automorphism group of the free Cantor algebra $C_n[r]$ of type $n$ on $r$ generators.

\paragraph{\textbf{Trees and forests.}} In this section, we give an alternative definition of these groups which relies on trees and forests.
A \emph{tree} is a finite connected directed acyclic graph such that each vertex is the target of at most one edge. 
We call \emph{internal vertices} the vertices which are not univalent.
We consider trees with one special vertex which is the source of one edge and not a target for any edge. 
This vertex is called the \emph{root} and the trees are then \emph{rooted trees}.
The other univalent vertices are called \emph{leaves}.
We will use the terminology root and leaf sometimes for the univalent vertices, sometimes for the edges connected to them.

A tree is \emph{$n$-ary} if each internal vertex is the source of $n$ edges.
Let a \emph{forest} be a disjoint union of trees.
An $r$-forest is a forest with $r$ connected components. An $(n,r)$-forest is a forest with $r$ roots and only $n$-ary trees.

\paragraph{\textbf{Paired forest diagrams.}} The following vocabulary and framework is introduced and used by Skipper and Wu \cite{skipperwu}.
A \emph{paired $(n,r)$-forest diagram} is a triple $(F_{-},\sigma, F_{+})$ consisting of two $(n,r)$-forests both with $l$ labeled leaves for some integer $l$, and $\sigma$ a permutation of $l$ elements. 

Let $F$ be an $(n,r)$-forest.
We call \emph{caret} an internal vertex of $F$ along with its $n$ descendant edges and vertices. 
A caret is \emph{elementary} if its leaves are leaves of the forest.

Let $(F_{-},\sigma, F_{+})$ be a paired $(n,r)$-forest diagram. Suppose there is an elementary caret in $F_{-}$
with leaves labeled by $i,\dots, i+n-1$, and an elementary caret in $F_{+}$ with leaves labeled by $\sigma(i),\dots,\sigma(i+n-1)$.
Then a \emph{reduction} of the diagram is obtained by removing both carets, renumbering the leaves, and replacing $\sigma$ with the permutation $\sigma'$ which sends the new leaf of $F_{-}$ to the new leaf of $F_{+}$ and is $\sigma$ elsewhere. 
The inverse operation of a reduction is called an \emph{expansion}.

    \begin{figure}
    \begin{minipage}{0.35\linewidth}
    \begin{center}
        \begin{forest}
            for tree = {s sep =  10pt, l-=5mm, inner sep = 1.5pt, if n children=0{tier=terminal}{}}
            [, phantom, s sep = 1cm,
            [, circle, fill = black, draw,
            [, circle, fill = black, draw,
              [, circle, label = {[font=\scriptsize]south:1}, fill = black, draw][, circle, label = {[font=\scriptsize]south:2}, fill = black, draw
              ][, circle, label = {[font=\scriptsize]south:3}, fill = black, draw]
              ]]
              [, circle, fill = black, draw,
            [, circle, fill = black, draw, 
              [, circle, fill = black, draw,
              [, circle, label = {[font=\scriptsize]south:4}, fill = black, draw][, circle, label = {[font=\scriptsize]south:5}, fill = black, draw][, circle, label = {[font=\scriptsize]south:6}, fill = black, draw]][, circle, label = {[font=\scriptsize]south:7}, fill = black, draw][, circle, label = {[font=\scriptsize]south:8},fill = black, draw]
             ]]
            ]
        \end{forest}
    \end{center}
    \begin{center}
        \begin{forest}
            for tree = {s sep =  10pt, l-=5mm, inner sep = 1.5pt, if n children=0{tier=terminal}{}}
            [, phantom, s sep = 1cm, grow=north,
            [, circle, fill = black, draw, grow=north,
            [, circle, fill = black, draw, grow=north,[, circle, label = {[font=\scriptsize]north:8},fill = black, draw][, circle, fill = black, draw, grow=north,
              [, circle, label = {[font=\scriptsize]north:6},fill = black, draw][, circle, label = {[font=\scriptsize]north:5},fill = black, draw][, circle, label = {[font=\scriptsize]north:4},fill = black, draw]
              ][, circle, label = {[font=\scriptsize]north:1},fill = black, draw]
              ]]
              [, circle, fill = black, draw, grow=north,
            [, circle, fill = black, draw, grow=north,
              [, circle, label = {[font=\scriptsize]north:7},fill = black, draw][, circle,label = {[font=\scriptsize]north:2}, fill = black, draw][, circle, label = {[font=\scriptsize]north:3},fill = black, draw]
              ]]
            ]
        \end{forest}
    \end{center}
    \end{minipage}
    $\rightarrow$
    \begin{minipage}{0.35\linewidth}
        \begin{center}
            \begin{forest}
                for tree = {s sep =  10pt, l-=5mm, inner sep = 1.5pt, if n children=0{tier=terminal}{}}
                [, phantom, s sep = 1cm,
                [, circle, fill = black, draw, 
                [, circle, fill = black, draw, 
                  [, circle, label = {[font=\scriptsize]south:1}, fill = black, draw][, circle, label = {[font=\scriptsize]south:2}, fill = black, draw
                  ][, circle, label = {[font=\scriptsize]south:3}, fill = black, draw]
                  ]]
                [, circle, fill = black, draw, 
                [, circle, fill = black, draw, 
                  [, circle, label = {[font=\scriptsize]south:4}, fill = black, draw][, circle, label = {[font=\scriptsize]south:5}, fill = black, draw][, circle, label = {[font=\scriptsize]south:6}, fill = black, draw]
                 ]]
                ]
            \end{forest}
        \end{center}
        \begin{center}
            \begin{forest}
                for tree = {s sep =  10pt, l-=5mm, inner sep = 1.5pt, if n children=0{tier=terminal}{}}
                [, phantom, s sep = 1cm, grow=north,
                [, circle, fill = black, draw, grow=north,
                [, circle, fill = black, draw, grow=north,[, circle, label = {[font=\scriptsize]north:6}, fill = black, draw][, circle, label = {[font=\scriptsize]north:4}, fill = black, draw, grow=north][, circle, label = {[font=\scriptsize]north:1}, fill = black, draw]
                  ]]
                [, circle, fill = black, draw, grow=north,
                [, circle, fill = black, draw, grow=north,
                  [, circle, label = {[font=\scriptsize]north:5}, fill = black, draw][, circle, label = {[font=\scriptsize]north:2}, fill = black, draw][, circle, label = {[font=\scriptsize]north:3}, fill = black, draw]
                  ]]
                ]
            \end{forest}
        \end{center}
    \end{minipage}
    \label{reduction}
    \caption{Example of a reduction of paired forest diagrams}
    \end{figure}
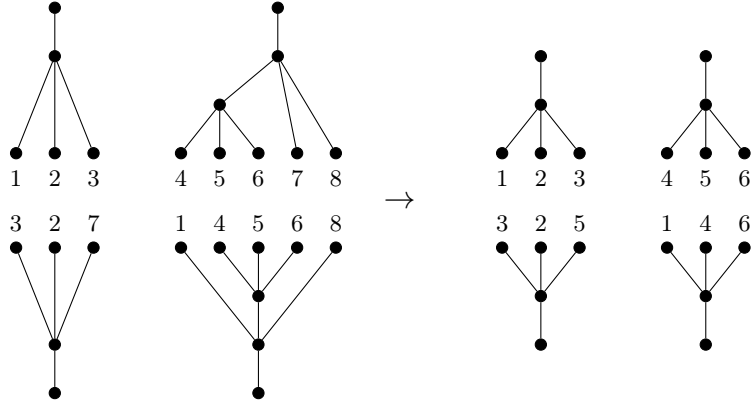

The multiplication of $(F,\sigma, F')$ and $(T,\tau, T')$ is given by taking representatives $(\tilde{F},\tilde{\sigma},\tilde{F}')$ and $(\tilde{T},\tilde{\tau},\tilde{T}')$ in the equivalence classes of these elements
so that $\tilde{F}$, $\tilde{F}'$, $\tilde{T}$ and $\tilde{T}'$ have the same number of leaves. Then $(F,\sigma, F')\cdot (T,\tau, T') = (\tilde{F},\tilde{\sigma}\circ\tilde{\tau}, \tilde{T}')$.
\begin{exemple}
    Let us take $r=1$ and $n=2$ in this example.
    \begin{center}
    \begin{minipage}{0.15\linewidth}
        \begin{center}
            \begin{forest}
                for tree = {s sep =  10pt, l-=5mm,inner sep = 1.5pt, if n children=0{tier=terminal}{}}
                [, circle, fill = black, draw,
                [, circle, fill = black, draw,
                  [, circle, fill = black, draw,
                  [, circle, label = {[font=\scriptsize]south:1}, fill = black, draw][, circle, label = {[font=\scriptsize]south:2}, fill = black, draw]
                  ][, circle, label = {[font=\scriptsize]south:3}, fill = black, draw]
                ]]
            \end{forest}
        \end{center}
        \begin{center}
            \begin{forest}
                for tree = {s sep =  10pt, l-=5mm,inner sep = 1.5pt, if n children=0{tier=terminal}{}}
                [, circle, fill = black, draw, grow=north,
                [, circle, fill = black, draw, grow=north,
                  [, circle, fill = black, draw,grow=north,
                  [, circle, label = {[font=\scriptsize]north:3}, fill = black, draw,grow=north][, circle, label = {[font=\scriptsize]north:2}, fill = black, draw,grow=north]
                  ][, circle, label = {[font=\scriptsize]north:1}, fill = black, draw,grow=north]
                ]]
            \end{forest}
        \end{center}
        \end{minipage}
        $\bullet$
        \begin{minipage}{0.15\linewidth}
            \begin{center}
                \begin{forest}
                    for tree = {s sep =  10pt, l-=5mm,inner sep = 1.5pt, if n children=0{tier=terminal}{}}
                    [, circle, fill = black, draw,
                    [, circle, fill = black, draw,
                      [, circle, fill = black, draw, label = {[font=\scriptsize]south:1}][, circle, label = {[font=\scriptsize]south:2}, fill = black, draw]
                    ]]
                \end{forest}
            \end{center}
            \begin{center}
                \begin{forest}
                    for tree = {s sep =  10pt, l-=5mm,inner sep = 1.5pt, if n children=0{tier=terminal}{}}
                    [, circle, fill = black, draw, grow= north,
                    [, circle, fill = black, draw, grow= north,
                      [, circle, fill = black, draw, label = {[font=\scriptsize]north:1}][, circle, label = {[font=\scriptsize]north:2}, fill = black, draw]
                    ]]
                \end{forest}
            \end{center}
        \end{minipage}
        $=$
        \begin{minipage}{0.15\linewidth}
            \begin{center}
                \begin{forest}
                    for tree = {s sep =  10pt, l-=5mm,inner sep = 1.5pt, if n children=0{tier=terminal}{}}
                    [, circle, fill = black, draw,
                    [, circle, fill = black, draw,
                      [, circle, fill = black, draw,
                      [, circle, label = {[font=\scriptsize]south:1}, fill = black, draw][, circle, label = {[font=\scriptsize]south:2}, fill = black, draw]
                      ][, circle, label = {[font=\scriptsize]south:3}, fill = black, draw]
                    ]]
                \end{forest}
            \end{center}
            \begin{center}
                \begin{forest}
                    for tree = {s sep =  10pt, l-=5mm,inner sep = 1.5pt, if n children=0{tier=terminal}{}}
                    [, circle, fill = black, draw, grow=north,
                    [, circle, fill = black, draw, grow=north,
                      [, circle, fill = black, draw,grow=north,
                      [, circle, label = {[font=\scriptsize]north:3}, fill = black, draw,grow=north][, circle, label = {[font=\scriptsize]north:2}, fill = black, draw,grow=north]
                      ][, circle, label = {[font=\scriptsize]north:1}, fill = black, draw,grow=north]
                    ]]
                \end{forest}
            \end{center}
            \end{minipage}
            $\bullet$
            \begin{minipage}{0.15\linewidth}
                \begin{center}
                    \begin{forest}
                        for tree = {s sep =  10pt, l-=5mm,inner sep = 1.5pt, if n children=0{tier=terminal}{}}
                        [, circle, fill = black, draw,
                        [, circle, fill = black, draw,
                          [, circle, label = {[font=\scriptsize]south:1}, fill = black, draw]
                          [, circle, fill = black, draw,
                          [, circle, label = {[font=\scriptsize]south:2}, fill = black, draw][, circle, label = {[font=\scriptsize]south:3}, fill = black, draw]
                          ]
                        ]]
                    \end{forest}
                \end{center}
                \begin{center}
                    \begin{forest}
                        for tree = {s sep =  10pt, l-=5mm,inner sep = 1.5pt, if n children=0{tier=terminal}{}}
                        [, circle, fill = black, draw, grow=north,
                        [, circle, fill = black, draw, grow=north,
                          [, circle, label = {[font=\scriptsize]north:1}, fill = black, draw,grow=north]
                          [, circle, fill = black, draw,grow=north,
                          [, circle, label = {[font=\scriptsize]north:3}, fill = black, draw,grow=north][, circle, label = {[font=\scriptsize]north:2}, fill = black, draw,grow=north]
                          ]
                        ]]
                    \end{forest}
                \end{center}
                \end{minipage}
                $=$
                \begin{minipage}{0.15\linewidth}
                    \begin{center}
                        \begin{forest}
                            for tree = {s sep =  10pt, l-=5mm,inner sep = 1.5pt, if n children=0{tier=terminal}{}}
                            [, circle, fill = black, draw,
                            [, circle, fill = black, draw,
                              [, circle, fill = black, draw,
                              [, circle, label = {[font=\scriptsize]south:1}, fill = black, draw][, circle, label = {[font=\scriptsize]south:2}, fill = black, draw]
                              ][, circle, label = {[font=\scriptsize]south:3}, fill = black, draw]
                            ]]
                        \end{forest}
                    \end{center}
                    \begin{center}
                        \begin{forest}
                            for tree = {s sep =  10pt, l-=5mm,inner sep = 1.5pt, if n children=0{tier=terminal}{}}
                            [, circle, fill = black, draw, grow=north,
                            [, circle, fill = black, draw, grow=north,
                              [, circle, label = {[font=\scriptsize]north:1}, fill = black, draw]
                              [, circle, fill = black, draw,grow=north,
                              [, circle, label = {[font=\scriptsize]north:3}, fill = black, draw][, circle, label = {[font=\scriptsize]north:2}, fill = black, draw]
                              ]
                            ]]
                        \end{forest}
                    \end{center}
                    \end{minipage}
                \end{center}
\end{exemple}

\begin{prop}\cite[Definition~3.1]{skipperwu}
    The Higman--Thompson group $V_{n,r}$ is isomorphic to the group of equivalence classes of paired $(n,r)$-forest diagrams, where two forest diagrams are equivalent if one can be obtained from the other through a finite series of reductions and expansions. 
    \end{prop}
    We can define other groups in the same family by considering paired forest diagrams where admissible permutations are in a subgroup of $\Sigma_n$ \cite{higman}. 
    In the family of groups $F_{n,r}$, the only admissible permutation is the identity; in the family $T_{n,r}$ the only permutations allowed are cycles.    

    We now need to mention a sort of periodicity property of the Higman--Thompson groups:
    \begin{prop}\cite[proof of Theorem~3.6]{szymikwahl}
        The groups $V_{n,r}$ and $V_{n,r+(n-1)}$ are isomorphic.
    \end{prop}
\subsection{Operad groups}
Thumann \cite{thumann} provides a framework where Higman--Thompson groups arise as the \textit{operad groups} of certain operads. 
We here recall this framework. 
We will rely on it to build a topological model for the Higman--Thompson groups, using Proposition \ref{concentratedone}.
A symmetric operad in a symmetric monoidal category $(\C,\otimes, u)$ \cite{may, boardmanvogt} is the data of a collection of objects of the category $\{\mathcal{O}(n)\}_{n\in\mathbb{N}}$ for every integer $n$, together with:
\begin{itemize}
    \item a morphism $\eta:u \rightarrow \mathcal{O}(1)$;
    \item an action of the symmetric group $\Sigma_n$ on each $\mathcal{O}(n)$;
    \item morphisms $\mathcal{O}(r)\otimes\mathcal{O}(k_1)\otimes\dots\otimes \mathcal{O}(k_r)\rightarrow \mathcal{O}(k_1+\dots k_r)$ for every integer $r$ and every $(k_1,\dots, k_r) \in\mathbb{N}^r$. 
\end{itemize}
We require these morphisms to be suitably associative, unital, and equivariant. 

There is a functor 
\[\End: \SymMon \rightarrow \Op\]
which associates to each symmetric monoidal category $\C$ an operad $\End(\C)$, where 
\[\End(\C)(n) = \Hom(\C^{\otimes n},\C).\]
An algebra over the operad $\op$ is given by a morphism of operads $\op\rightarrow \End(\C)$.
The functor $\End$ has a left adjoint:
\[\s:\Op\rightarrow \SymMon.\] 
The category $\s(\op)$ associated to an operad $\op$ is called its PROP.
Let us give another, more explicit, definition of this category.
The PROP $\s(\op)$ associated to an operad $\op$ has object space $\mathbb{N}$. Morphisms from $m$ to $n$ are sequences of $n$ operations in $\op$, with total number of inputs $m$.

\begin{defi}
Let $\op$ be a symmetric operad, and $r$ be an object in $\s(\op)$. The group $\pi_1(\op,r) = \pi_1(\vert\s(\op)\vert, r)$ is the \emph{operad group} associated to $\op$ and based at the object $r$.
\end{defi}

\subsection{Cancellative calculus of fractions}
We will now recall briefly the cancellative calculus of fractions, first in the case of categories and then in the case of operads.

\begin{defi}
    A category $\C$ satisfies the calculus of fractions if the following properties are satisfied:
    \begin{itemize}
        \item \textit{Square filling.} For every pair of morphisms $f:B\rightarrow A$ and $g:C\rightarrow A$ there are morphisms $a:D\rightarrow B$ and $b:D\rightarrow C$ such that $fa=gb$.
        \begin{center}
        \begin{tikzcd}
        D\arrow[r,dotted,"\exists a"] \arrow[d,dotted, "\exists b"] &B\arrow{d}{f}\\
        C\arrow{r}{g} &A.
        \end{tikzcd}
        \end{center}
        \item \textit{Equalization.} For all morphisms $f,g:A\rightarrow B$ and $a:B\rightarrow C$ such that $af=ag$, then there exists an arrow $b:D\rightarrow A$ with $fb=gb$.
        \begin{center}
        \begin{tikzcd}
        D\arrow[r,dotted,"\exists b"] &A\arrow[r,swap,"f"]\arrow[r,shift left,"g"] &B\arrow[r,"a"]&C.
        \end{tikzcd}
        \end{center}
    \end{itemize}
\end{defi}

\begin{defi}
    A category $\C$ is \emph{right cancellative} if for morphisms $f,g,a$, $af=ag$ implies $f=g$. It is \emph{left cancellative} if for morphisms $f,g,a$, $fa=ga$ implies $f=g$.
    It is called \emph{cancellative} if it is both right and left cancellative.
\end{defi}

We can define an analog notion for operads.
\begin{notat}
Let $\theta,\psi$ be two operations in an operad. Let us denote by $\theta\approx\psi$ the fact that $\theta$ and $\psi$ are equivalent modulo the action of the symmetric groups; i.e. there exists a permutation $\gamma$ such that $\theta = \gamma\cdot\psi.$
\end{notat}

\begin{defi}
Let $\op$ be a symmetric operad. We say that $\op$ satisfies the calculus of fractions if the following conditions are satisfied:
\begin{itemize}
\item \textit{Square filling.} For every pair of operations $\theta_1$ and $\theta_2$, there are sequences of operations $\Psi_1 = (\psi_1^1,\dots,\psi_1^{k_1})$ and $\Psi_2 = (\psi_2^1,\dots,\psi_2^{k_2})$ such that $\theta_i*\Psi_i$ is defined for $i=1,2$ and such that $\theta_1 * \Psi_1\approx \theta_2 * \Psi_2$.
\item \textit{Equalization.} Let $\theta$ be an operation and $\Psi_1 = (\psi_1^1,\dots,\psi_1^{k})$ and $\Psi_2 = (\psi_2^1,\dots,\psi_2^{k})$ be sequences of operations such that $\theta*\Psi_1 \approx \theta*\Psi_2$, which means that there is a $\gamma$ such that $\theta*\Psi_1 = \gamma\cdot(\theta*\Psi_2).$ Then $\gamma$ is of the form $\gamma=\gamma_1\otimes\dots\otimes\gamma_k$ where $\gamma_j\cdot\psi_2^j$ is defined for every $1\leq j \leq k$ and there is a sequence of operations $\Xi_j$ for every $1\leq j \leq k$ such that $\psi^j_1*\Xi_j = (\gamma_j\cdot\psi_2^j)*\Xi_j.$
\end{itemize}
\end{defi}

\begin{defi}
Let $\op$ be a symmetric operad. We say that $\op$ is \emph{right cancellative} if the following is verified:
let $\theta$ be an operation and $\Psi_1 = (\psi_1^1,\dots,\psi_1^{k})$, $\Psi_2 = (\psi_2^1,\dots,\psi_2^{k})$ be sequences of operations such that $\theta*\Psi_1 \approx \theta*\Psi_2$, which means that there is a $\gamma$ such that $\theta*\Psi_1 = \gamma\cdot(\theta*\Psi_2).$ Then $\gamma$ is of the form $\gamma=\gamma_1\otimes\dots\otimes\gamma_k$ where $\gamma_j\cdot\psi_2^j$ is defined for every $1\leq j \leq k$ and $\gamma_j\cdot\psi_2^j = \psi_1^j$ for every $1\leq j \leq k.$
\end{defi}

\begin{defi}
Let $\op$ be a symmetric operad. We say that $\op$ is \emph{left cancellative} if the following is verified:
let $\theta_1, \theta_2$ be operations and $\Psi$ a sequence of operations such that $\theta_1*\Psi = \theta_2*\Psi$. Then $\theta_1 = \theta_2$. 
\end{defi}
A symmetric operad $\op$ is \emph{cancellative} if it is both right and left cancellative.

\begin{prop}\label{cancellativeprop}\cite[Proposition~3.9-3.10]{thumann}
An operad $\op$ satisfies the calculus of fractions if and only if its associated PROP $\s(\op)$ does. 
An operad $\op$ is cancellative if and only if its associated PROP $\s(\op)$ is. 
\end{prop}

\begin{prop}\label{concentratedone}\cite[Proposition~2.13]{thumann}
    If a category $\C$ satisfies the cancellative calculus of fractions, then the functor $\C\rightarrow \Pi_1(\C)$ to its fundamental groupoid is a homotopy equivalence.
\end{prop}


\subsection{Cube cutting operads}
We now consider only operads that are \emph{reduced}, i.e. that only have one operation of arity 1.
Thumann \cite{thumann} exhibits the Higman--Thompson groups as operad groups of certain operads, called the cube cutting operads.
We are going to look at suboperads of the endomorphism operad of the symmetric monoidal category $(\Top, \sqcup)$ in the sense of the following definition.
\begin{defi}
    Let $\C$ be a category and $S$ a set of operations of degree at least 2 in $\End(\C)$.
    Then the suboperad of $\End(\C)$ generated by this data is the smallest suboperad $\op$ of $\End(\C)$ which contains $S$.
\end{defi}

Let $N$ be a finite set of natural numbers greater than or equal to 2. 
The elements $n_1,\dots, n_k$ of $N$ are said to be \emph{independent} if $n_1^{r_1}\dots n_k^{r_k} = n_1^{s_1}\dots n_k^{s_k}$ implies $s_1 = r_1, \dots, r_k = s_k$.
Note that if $n_1,\dots, n_k$ are pairwise coprime, then they are independent.
However, the converse is not true: $2$ and $6$ are not coprime but they are independent.

For $k\geq 1$, consider the $k$-dimensional unit cube.
For each $j\in{1,\dots, k}$, let $N_j$ be a set of independent numbers.
For every $j$, and $n\in N_j$, a \emph{very elementary subdivision} of the cube is given by cutting it, along the $j$th coordinate, into $n$ pieces. The \emph{very elementary operation} associated to such a subdivision is the rescaling of $n$ cubes to fit into the pieces of the subdivision.
The set $S$ of all the very elementary operations, obtained for $1\leq j \leq k$ and $n\in N_j$, generates a suboperad $\op_{k,N_1,\dots, N_k}$ of $\End(\Top,\sqcup)$, called the \emph{cube cutting operad}.

\begin{exemple}
    The following figure illustrates the operadic composition in the case where $k=2$, $N_1 = \{2\}$ and $N_2 = \{3\}$:
    \begin{center}
        \begin{minipage}{0.25\linewidth}
            \begin{center}
            \begin{tikzpicture}
            \draw (0,0)--(0,2);
            \draw (0,0)--(2,0);
            \draw (2,2)--(2,0);
            \draw (2,2)--(0,2);
            \draw (1,0) -- (1,2);
            \node at (0.5,0.8) [label = 1] {};
            \node at (1.5,0.8) [label = 2] {};
            \end{tikzpicture}
            \end{center}
            \end{minipage}
            $\circ_2$
            \begin{minipage}{0.25\linewidth}
            \begin{center}
            \begin{tikzpicture}
                \draw (0,0)--(0,2);
                \draw (0,0)--(2,0);
                \draw (2,2)--(2,0);
                \draw (2,2)--(0,2);
                \draw (0,0.67) -- (2,0.67);
                \draw (0,1.33) -- (2,1.33);
                \node at (1,0) [label = 1] {};
            \node at (1,1.2) [label = 2] {};
            \node at (1,0.6) [label = 3] {};
            \end{tikzpicture}
            \end{center}
            \end{minipage}
            $=$
            \begin{minipage}{0.25\linewidth}
            \begin{center}
            \begin{tikzpicture}
                \draw (0,0)--(0,2);
                \draw (0,0)--(2,0);
                \draw (2,2)--(2,0);
                \draw (2,2)--(0,2);
                \draw (1,0) -- (1,2);
                \node at (0.5,0.8) [label = 1] {};
                \draw (1,0.67) -- (2,0.67);
                \draw (1,1.33) -- (2,1.33);
                \node at (1.5,0) [label = 2] {};
            \node at (1.5,1.2) [label = 3] {};
            \node at (1.5,0.6) [label = 4] {};
            \end{tikzpicture}
            \end{center}
            \end{minipage}
        \end{center}
\end{exemple}

\begin{prop}\label{cancellativeoperad}
    For $k\in\mathbb{N}$ and $N_j\subset\mathbb{N}$ sets of independent integers for every $1\leq j\leq k$, the associated cube cutting operad satisfies the cancellative calculus of fractions.
\end{prop}

\begin{prop}
    The Higman--Thompson groups $V_{n,r}$ are isomorphic to the operad groups associated to the cube cutting operads $\op_{1,n}$ (where $k=1$ and $N = \{n\}$), based at the object $r$.
\end{prop}

The following proposition relates the classifying space of the groups to that of the PROP associated to the cube cutting operad.
It is in some sense the reason for introducing operads in the first place, as it enables us to now focus our study on the category $\s(\op_{1,n})$.
\begin{prop}\label{propgroups}
    There is an equivalence 
    \[\displaystyle \bigsqcup_{0\leq r\leq n-1} BV_{n,r}\simeq B\s(\op_{1,n}).\]
\end{prop}

\begin{proof}
    The operad $\op_{1,n}$ has cancellative calculus of fractions (proposition \ref{cancellativeoperad}), so its PROP does as well by proposition \ref{cancellativeprop}. 
    We can therefore apply proposition \ref{concentratedone} to compare the prop to its fundamental groupoid.
    \begin{align*}
        B\s(\op_{1,n}) &\simeq B\Pi_1(\s(\op_{1,n})) \simeq \displaystyle \bigsqcup_{r<n} B\Aut_{\s(\op_{1,n})}(r) \simeq \displaystyle \bigsqcup_{r<n} BV_{n,r}.\qedhere
    \end{align*}
\end{proof}

\section{Embedding spaces of trees}\label{embeddings}
We will consider in the next section a topological version of the PROPs associated to the cube cutting operads.
To that end, we need to consider certain spaces of embeddings of forests inside $I^N$, where $N$ is large and potentially infinite.
The main result of this section is that our space of embeddings of a forest $T$ inside $I^\infty$ is contractible.
We now think of a forest as a finite CW-complex of dimension 1. 
The 0-cells are the vertices, including roots and leaves, which are the univalent vertices. 

\begin{rema}
    We fix an angle for the tree embeddings which will remain the same throughout. 
    The value of this angle $\alpha$ does not really matter as long as $\frac{\pi}{3n}\leq\alpha\leq \frac{\pi}{2n}$.
    We fix an affine embedding $\phi_\alpha(\epsilon)$ of a $n$-ary corolla (see figure \ref{standardimage}) with fixed angles $\alpha$ embedded in an $(N+1)$-dimensional cube $C_\epsilon$ of size $2\epsilon$ for every $\epsilon >0$. 
\end{rema}

\begin{defi}\label{embeddingspaces}
Let $N \in\mathbb{N}\cup\{\infty\}$ and $T$ be a rooted $n$-ary forest. We consider embeddings of $T$ into $I^{N}\times\mathbb{R}$, i.e. continuous maps which are homeomorphisms onto their image.
For such an embedding $\phi$, let $\tilde{\phi}$ be the projection on the last coordinate. We will sometimes refer to this coordinate as the time dimension.
Let $\E_N(T)$, for $N$ in $\mathbb{N}\cup\{\infty\}$, be the space of pairs $(\phi,\epsilon)$ where $\phi$ is an embedding of $T$ into $I^{N}\times\mathbb{R}$ and $\epsilon>0$ such that:
\begin{itemize}
    \item[(i)] $\tilde{\phi}$ is linear and increasing on each 1-cell;
    \item[(ii)] $\tilde{\phi}$ is constant on the set of leaves, and also on the set of roots;
    \item[(iii)] for every internal vertex $v$ of $T$, there exists $V$ a neighborhood of $v$ such that $\phi\vert_V$ coincides with $\phi_\alpha(\epsilon)$ and such that $\im(\phi)\cap C_\epsilon$ is only the embedded corolla; the cubes $C_\epsilon$ are called \emph{vertex cubes}.
\end{itemize}
\end{defi}

A point $t$ in $\mathbb{R}$ is called a junction if it is the image of an internal vertex by $\tilde{\phi}$.

\begin{figure}
    \centering
    \def\svgwidth{\hsize}
    \begingroup%
  \makeatletter%
  \providecommand\color[2][]{%
    \errmessage{(Inkscape) Color is used for the text in Inkscape, but the package 'color.sty' is not loaded}%
    \renewcommand\color[2][]{}%
  }%
  \providecommand\transparent[1]{%
    \errmessage{(Inkscape) Transparency is used (non-zero) for the text in Inkscape, but the package 'transparent.sty' is not loaded}%
    \renewcommand\transparent[1]{}%
  }%
  \providecommand\rotatebox[2]{#2}%
  \newcommand*\fsize{\dimexpr\f@size pt\relax}%
  \newcommand*\lineheight[1]{\fontsize{\fsize}{#1\fsize}\selectfont}%
  \ifx\svgwidth\undefined%
    \setlength{\unitlength}{395.28769666bp}%
    \ifx\svgscale\undefined%
      \relax%
    \else%
      \setlength{\unitlength}{\unitlength * \real{\svgscale}}%
    \fi%
  \else%
    \setlength{\unitlength}{\svgwidth}%
  \fi%
  \global\let\svgwidth\undefined%
  \global\let\svgscale\undefined%
  \makeatother%
  \begin{picture}(1,0.49525469)%
    \lineheight{1}%
    \setlength\tabcolsep{0pt}%
    \put(0,0){\includegraphics[width=\unitlength,page=1]{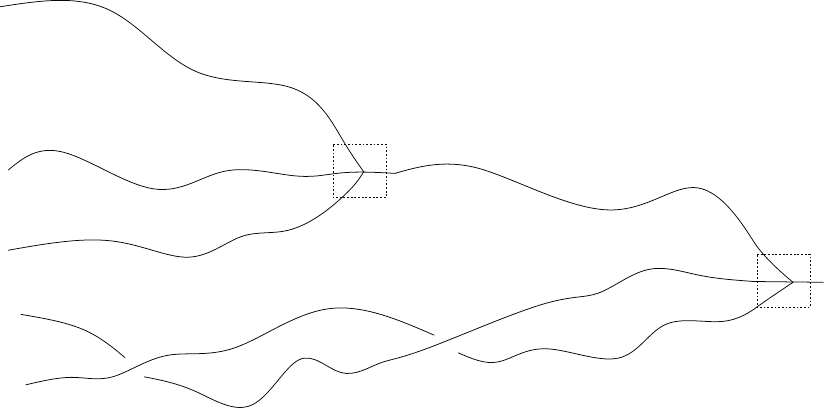}}%
    \put(0.40665496,0.34698933){\color[rgb]{0,0,0}\makebox(0,0)[lt]{\lineheight{1.25}\smash{\begin{tabular}[t]{l}$\epsilon$\end{tabular}}}}%
    \put(0,0){\includegraphics[width=\unitlength,page=2]{plongement.pdf}}%
  \end{picture}%
\endgroup%

    \caption{Example of an element in $\E_N(T)$}
    \label{standardimage}
\end{figure}


\subsection{Holey configurations}



A useful result that we need is the connectivity of configuration spaces with some holes in them:
\begin{lem}\label{connectedness}
    Let $k,N,m$ be integers. Let $B_0,\dots, B_m$ be a collection of disjoint open cubes whose disjoint union is strictly included inside $I^N$. 
    The space $\Conf(k,I^N\backslash \sqcup B_i)$ is $(N-2)$-connected.
\end{lem}

\begin{proof}
    
    
   
    The space $\Conf(k,I^N\backslash \sqcup B_i)$ deformation retracts onto $\Conf(k,\mathbb{R}^N\backslash \{p_0,\dots, p_m\})$, so let us study the space $\Conf(k,\mathbb{R}^N\backslash \{p_0,\dots, p_m\}).$
    Recall that for $m=0$, and any $k,N$, the space $\Conf(k,\mathbb{R}^N)$ is $(N-2)$--connected. 
    We will show by induction on $k$ that for every $m$, the space $\Conf(k,\mathbb{R}^N\backslash \{p_0,\dots, p_m\})$ is $(N-2)$-connected.
    This is clearly true for $k=1$ and any $m$, because $\mathbb{R}^N\backslash \{p_0,\dots, p_m\}$ is equivalent to a wedge of $(N-1)$-spheres, and as such $(N-2)$--connected. 
    Let us take $k>1$ such that this hypothesis is verified, i.e. $\Conf(k,\mathbb{R}^N\backslash \{p_0,\dots, p_m\})$ is $(N-2)$--connected for all $m$.
    There is a fiber sequence:
    \begin{center}
        \begin{tikzcd}
            \Conf(k-1,\mathbb{R}^N\backslash \{p_0,\dots, p_m,p_{m+1}\})\arrow{d}\\
            \Conf(k,\mathbb{R}^N\backslash \{p_0,\dots, p_m\}) \arrow{d}\\
            \Conf(1,\mathbb{R}^N\backslash \{p_0,\dots, p_m\}).
        \end{tikzcd}
    \end{center}
    The top space is $(N-2)$--connected by induction hypothesis, the bottom space is homeomorphic to a wedge of $(N-1)$-spheres so it is $(N-2)$--connected, so the middle space is $(N-2)$--connected.
    This concludes our proof by induction.
\end{proof}

Now let us fix a forest $T$ until the end of the subsection, and decompose our space $\E_N(T)$, by introducing the space $\V_N(T)$ of possible configurations for the vertices of a given forest $T$. 
\begin{defi}
    Let $\V_N(T)$ be the subspace of configurations of points inside $\mathbb{R}^N\times\mathbb{R}$ which are the vertex images of a certain element $(\phi,\epsilon)$ of $\E_N(T)$.
\end{defi}

There is a partial order on the vertices of $T$ induced by the structure of the forest. 
Let $V$ be the poset of internal vertices of $T$.

Let $\Pos$ be the category of posets. We then denote by $\Pos(P,Q)$ the order-preserving maps from $P$ to $Q$. 
There is a map $\rho: \V_N(T)\rightarrow \Pos(V,\mathbb{R})$ which sends a configuration of points in $\mathbb{R}^N\times\mathbb{R}$ to their projection on the time dimension.

\begin{lem}
    The space $\V_\infty(T)$ is contractible.
\end{lem}

\begin{proof}
The space $\V_N(T)$ is an open subset of $\Conf(\vert V\vert,I^N\times\mathbb{R})$, which is itself an open subset of $(I^N\times\mathbb{R})^{\vert V\vert}$.
The map $\rho$ is the restriction of the projection $(I^N\times\mathbb{R})^{\vert V\vert}\rightarrow \mathbb{R}^{\vert V \vert}$ to the open subspace $\V_N(T)$, so it is a microfibration.

The fiber over a given poset map $p$ is given by configurations of points in $I^N$ above each time point which is the image of one (or multiple) vertices.
This means that the fiber can be directly expressed as $\displaystyle \prod_{t\in\mathbb{R}} \Conf(\vert p^{-1}(t)\vert,I^N)$.
This is a product of $(N-2)$-connected spaces so it is $(N-2)$-connected.
When $N$ goes to infinity, this fiber is contractible.
The map from $\V_\infty(T)$ to  $\Pos(V,\mathbb{R})$ is thus a microfibration with contractible fibers and hence a weak equivalence.

The space $\Pos(V,\mathbb{R})$ is an intersection of convex subspaces of $\mathbb{R}^{\vert V\vert}$ so it is itself convex (and again non empty), and as such contractible.
It follows that the space $\V_\infty(T)$ is contractible.
\end{proof}

We would like to study the map $\E_N(T)\rightarrow \V_N(T)$.
This map is a microfibration for every $N$.
Let us call $\F_N(T)$ the fiber above a configuration of vertex images $(v_i)\in\V_N(T)$.
In other words, the $(v_i)$ specify the position of the vertices and we just need to connect them together in the shape of $T$ to get an element in $\F_N(T)$. 

Let us introduce an intermediary space $\G_N(T)$ which contains all elements of $\F_N(T)$ such that the cubes are far enough away from each other.
\begin{defi}
For a given configuration of vertex images $(v_i)$ with time projections $t_i$, let us set $\epsilon_{\aaaaah} = \underset{i\neq j}{\min}(\frac 13 \vert t_i - t_j\vert,\dist(v_i,\partial I^N))$.
Then we define $\G_N(T)$ to be the subspace of $\F_N(T)$ containing those $(\phi,\epsilon)$ where $\epsilon\leq\epsilon_{\aaaaah}$.
\end{defi}
The condition on $\epsilon$ in the previous definition ensures that for every element in $\G_N(T)$, the only cubes whose time dimensions overlap are those whose centers have exactly the same time projection. 
(We also ensure that the disjoint union of all the cubes is strictly included inside $I^N$).

\begin{lem}\label{smallcubes}
The space $\G_N(T)$ is a strong deformation retract of $\F_N(T)$.
\end{lem}

\begin{proof}
    There is a retraction $r:\F_N(T)\rightarrow \G_N(T)$ which sends $(\phi,\epsilon)$ to $\left(\phi,\min(\epsilon, \epsilon_{\aaaaah})\right)$.
    There is also an inclusion $i:\G_N(T)\hookrightarrow \F_N(T)$.
    The composition $r\circ i$ is the identity, and $i\circ r$ is homotopic to the identity via a linear shrinking of the cubes.
\end{proof}

\begin{lem}\label{Gconnected}
    The space $\G_N(T)$ is $(N-3)$-connected.
\end{lem}

\begin{proof}
 The space $\G_N(T)$ can be expressed as an iterated pullback $W_0\times_{C_0}W'_0\times_{C'_0}\dots\times_{C'_{k-1}}W_{k}$, where $k$ is the number of cubes with disjoint time projections (see figure \ref{numberofcubes}). 
Very broadly, the $W_j$ and $W'_j$ are spaces of paths of configurations and the $C_j$ and $C'_j$ are the configuration spaces where the paths connect to each other via target and source maps (see figures \ref{numberofcubes} and \ref{WandC}).

Let us first describe the spaces $C_j$ and $C'_j$ in detail. We need to establish some notation. 
Recall that $T$ is a forest with $l$ leaves and $r$ roots.
The space $\G_N(T)$ is a subspace of $\F_N(T)$, the fiber above a configuration $(v_i)$ of images for the vertices of the tree.
We consider the set $\{t_0,\dots,t_{k-1}\}$, which are the images of the vertices (with no multiplicity).
Let $n_j$ be the number of cubes which ended before (or at) $t_{j}$, and $m_j$ the number of cubes whose time projection contains $t_{j}$. 
Let $b_j$ be the number of paths arriving at the point $t_j$: using our notation, $b_j = l-(n-1)n_j$ (see figure \ref{numberofcubes}).

\begin{figure}[h!]
    \def\svgwidth{0.8\hsize}
    \begingroup%
  \makeatletter%
  \providecommand\color[2][]{%
    \errmessage{(Inkscape) Color is used for the text in Inkscape, but the package 'color.sty' is not loaded}%
    \renewcommand\color[2][]{}%
  }%
  \providecommand\transparent[1]{%
    \errmessage{(Inkscape) Transparency is used (non-zero) for the text in Inkscape, but the package 'transparent.sty' is not loaded}%
    \renewcommand\transparent[1]{}%
  }%
  \providecommand\rotatebox[2]{#2}%
  \newcommand*\fsize{\dimexpr\f@size pt\relax}%
  \newcommand*\lineheight[1]{\fontsize{\fsize}{#1\fsize}\selectfont}%
  \ifx\svgwidth\undefined%
    \setlength{\unitlength}{371.57375065bp}%
    \ifx\svgscale\undefined%
      \relax%
    \else%
      \setlength{\unitlength}{\unitlength * \real{\svgscale}}%
    \fi%
  \else%
    \setlength{\unitlength}{\svgwidth}%
  \fi%
  \global\let\svgwidth\undefined%
  \global\let\svgscale\undefined%
  \makeatother%
  \begin{picture}(1,0.89015812)%
    \lineheight{1}%
    \setlength\tabcolsep{0pt}%
    \put(0,0){\includegraphics[width=\unitlength,page=1]{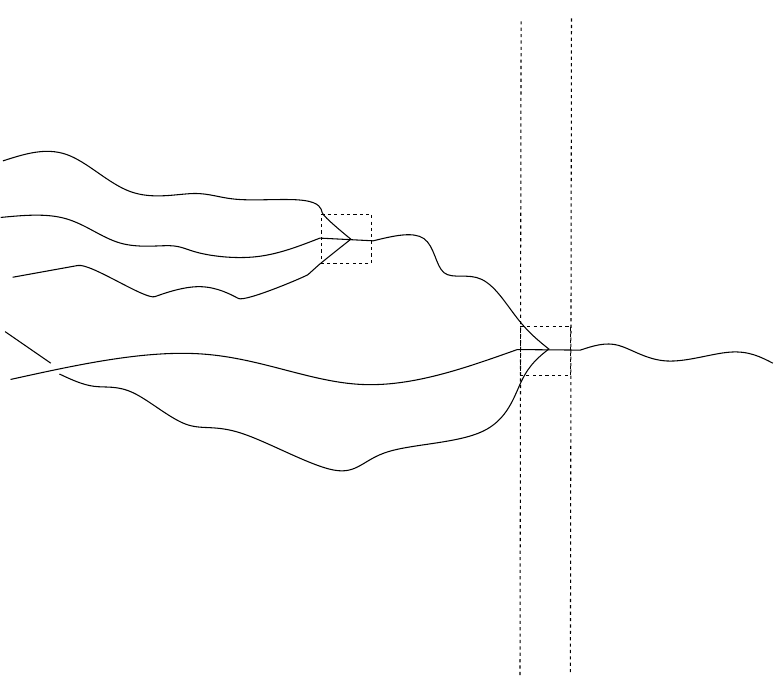}}%
    \put(0.62385238,0.0075474){\color[rgb]{0,0,0}\makebox(0,0)[lt]{\lineheight{1.25}\smash{\begin{tabular}[t]{l}$t_j-\epsilon$\end{tabular}}}}%
    \put(0.70962118,0.00453506){\color[rgb]{0,0,0}\makebox(0,0)[lt]{\lineheight{1.25}\smash{\begin{tabular}[t]{l}$t_j+\epsilon$\end{tabular}}}}%
    \put(0,0){\includegraphics[width=\unitlength,page=2]{numberofcubes.pdf}}%
    \put(0.1281371,0.86530607){\color[rgb]{0,0,0}\makebox(0,0)[lt]{\lineheight{1.25}\smash{\begin{tabular}[t]{l}$n_j$ cubes in this zone\end{tabular}}}}%
    \put(0.58639318,0.8635415){\color[rgb]{0,0,0}\makebox(0,0)[lt]{\lineheight{1.25}\smash{\begin{tabular}[t]{l}$m_j$ cubes in this zone\end{tabular}}}}%
    \put(0,0){\includegraphics[width=\unitlength,page=3]{numberofcubes.pdf}}%
    \put(0.3729565,0.69519216){\color[rgb]{0,0,0}\makebox(0,0)[lt]{\lineheight{1.25}\smash{\begin{tabular}[t]{l}$b_j$ paths arriving at $t_j-\epsilon$\end{tabular}}}}%
  \end{picture}%
\endgroup%
    \caption{An element in $\G_N(T)$}
    \label{numberofcubes}
\end{figure}

\begin{figure}[h!]
    \def\svgwidth{\hsize}
    \begingroup%
  \makeatletter%
  \providecommand\color[2][]{%
    \errmessage{(Inkscape) Color is used for the text in Inkscape, but the package 'color.sty' is not loaded}%
    \renewcommand\color[2][]{}%
  }%
  \providecommand\transparent[1]{%
    \errmessage{(Inkscape) Transparency is used (non-zero) for the text in Inkscape, but the package 'transparent.sty' is not loaded}%
    \renewcommand\transparent[1]{}%
  }%
  \providecommand\rotatebox[2]{#2}%
  \newcommand*\fsize{\dimexpr\f@size pt\relax}%
  \newcommand*\lineheight[1]{\fontsize{\fsize}{#1\fsize}\selectfont}%
  \ifx\svgwidth\undefined%
    \setlength{\unitlength}{728.12218373bp}%
    \ifx\svgscale\undefined%
      \relax%
    \else%
      \setlength{\unitlength}{\unitlength * \real{\svgscale}}%
    \fi%
  \else%
    \setlength{\unitlength}{\svgwidth}%
  \fi%
  \global\let\svgwidth\undefined%
  \global\let\svgscale\undefined%
  \makeatother%
  \begin{picture}(1,0.36991591)%
    \lineheight{1}%
    \setlength\tabcolsep{0pt}%
    \put(0,0){\includegraphics[width=\unitlength,page=1]{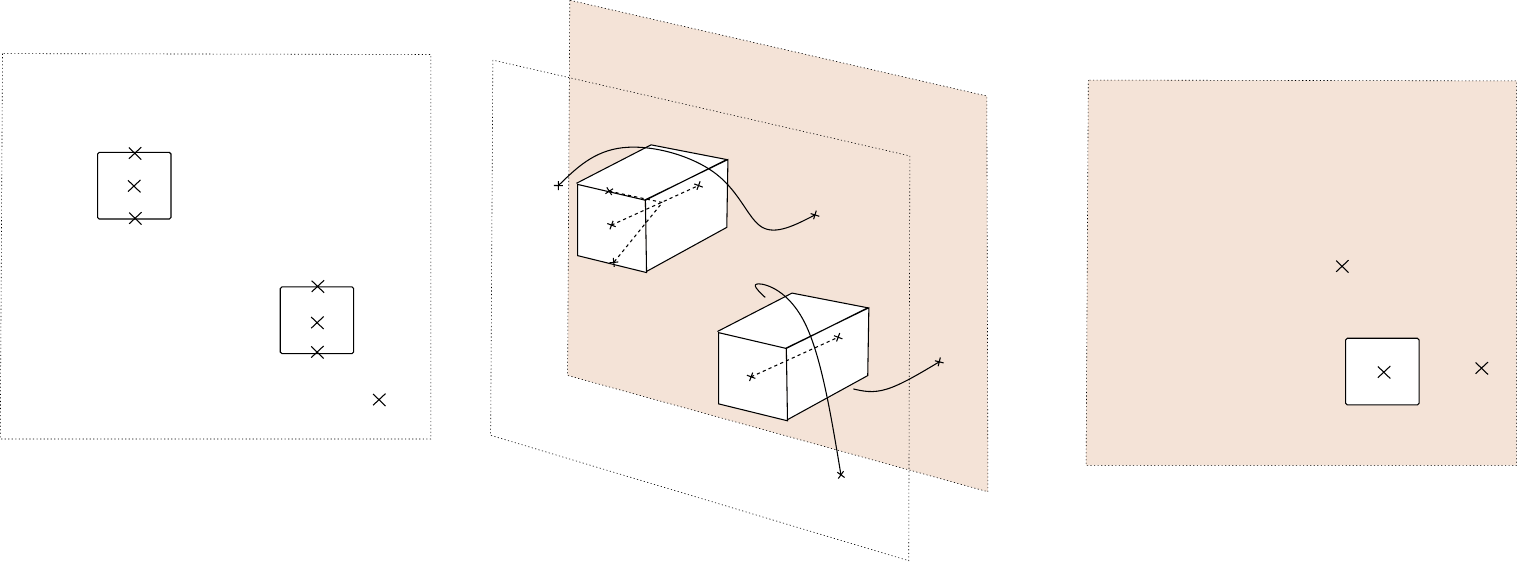}}%
    \put(0.48285162,0.35711618){\color[rgb]{0,0,0}\makebox(0,0)[lt]{\lineheight{1.25}\smash{\begin{tabular}[t]{l}$W'_j$\end{tabular}}}}%
    \put(0.81065766,0.34370593){\color[rgb]{0,0,0}\makebox(0,0)[lt]{\lineheight{1.25}\smash{\begin{tabular}[t]{l}$C'_j$\end{tabular}}}}%
    \put(0.16100567,0.35488114){\color[rgb]{0,0,0}\makebox(0,0)[lt]{\lineheight{1.25}\smash{\begin{tabular}[t]{l}$C_j$\end{tabular}}}}%
    \put(0,0){\includegraphics[width=\unitlength,page=2]{WandC.pdf}}%
  \end{picture}%
\endgroup%
    \caption{Illustration of the spaces $C_{j}$, $C'_{j}$ and $W'_{j}$, for $j$ even}
    \label{WandC}
\end{figure}

\begin{itemize}
\item The space $C_{j}$ is the space at the "beginning" of the cubes: an element in $C_{j}$ is a pair $(\epsilon, x)$ where $\epsilon\leq\epsilon_{\aaaaah}$ and $x$ is a configuration in $\Conf(b_{j}, I^N)$ with $n$ points marked off specifically inside each of $m_{j}$ cubes of size $2\epsilon$.
\item The space $C'_{j}$ is where cubes end: an element in $C'_{j}$ is a pair $(\epsilon, x)$ where $\epsilon\leq\epsilon_{\aaaaah}$ and $x$ is a configuration in $\Conf(b_{j}, I^N)$ with one point marked off in the center of each of $m_{j}$ cubes of size $2\epsilon$ (see figure \ref{WandC}).
\end{itemize}

Let us now describe the spaces $W_j$ and $W'_j$ in detail (see figure \ref{WandC}).
In full generality, note that a space of paths $I\rightarrow X\times I$ which is the identity on $I$ is equivalent to the space of paths $I\rightarrow X$.
We use this fact liberally below to make the proof (somewhat) more readable.

\begin{itemize}
\item First, for $k=0$, the space $W_0$ is the subspace of $\mathbb{R}\times P^{\Moore}(\Conf(l,I^N))$ consisting of those pairs 
$(\epsilon, \gamma)$ such that $\gamma:[t,t_0-\epsilon] \rightarrow \Conf(l,I^N)$ and $\epsilon\leq\epsilon_{\aaaaah}$.
The case of $W_k$ is similar.

\item The space $W'_{j}$ for $0\leq j\leq k-1$ is the space of paths of configurations that have to avoid the open cubes of size $2\epsilon$ centered around those $v_i$ whose time projection is $t_j$ (we write $v_i=(p_i,t_j)$). 
Therefore, $W'_{j}$ is equivalent to the subspace of $\mathbb{R}\times P^{\Moore}(\Conf(b_{j}-nm_{j},I^N\backslash \sqcup_i \{p_i\}))$ consisting of those pairs 
$(\epsilon, \gamma)$ such that $\epsilon\leq\epsilon_{\aaaaah}$ and \[\gamma:[t_j-\epsilon,t_j+\epsilon] \rightarrow \Conf(b_{j}-nm_{j},I^N\backslash \sqcup B_\epsilon),\] where the $B_\epsilon$ are cubes of size $2\epsilon$.
\item The space $W_{j}$ (for $1\leq j \leq k-1$) is a space of paths of configurations in $I^N$ ("between" the cubes). 
$W_{j}$ is the subspace of $\mathbb{R}\times P^{\Moore}(\Conf(b_{j},I^N))$ consisting of those pairs $(\epsilon, \gamma)$ such that \[\gamma:[t_{j-1}+\epsilon, t_j-\epsilon] \rightarrow \Conf(b_{j},I^N)\] (and such that $\epsilon\leq\epsilon_{\aaaaah}$).
\end{itemize}

The map $W_{j}\rightarrow C'_{j-1}$ (for $j\geq 1$) is a source map and the map $W_{j}\rightarrow C_{j}$ is a target map. These maps are fibrations.
The map $W'_{j}\rightarrow C_{j}$ is the composition of the source map to $\mathbb{R}\times \Conf(b_{j},I^N\backslash \sqcup_i \{p_i\})$ and the inclusion into $C_{j}$ which adds $n$ points around each of the $p_i$, evenly spaced inside a cube of size $2\epsilon$.
As the composition of a fibration and an equivalence, this map is also a fibration.
The map $W'_{j}\rightarrow C'_{j}$ is similarly the composition of a target map and an inclusion.
Therefore, every map in the diagram is a fibration.

Each $C_j$ (resp. $C'_j$) is equivalent to the configuration space of $I^N$ from which we have removed special cubes around the vertex points,
i.e. the space $\bigcup_\epsilon\Conf(b_j, I^N\backslash \sqcup B_\epsilon)$.
Each space $\Conf(b_j, I^N\backslash \sqcup B_\epsilon)$ is $(N-2)$-connected by lemma \ref{connectedness}.
Therefore, $C_j$ and $C'_j$ are $(N-2)$-connected for all $j$.

Let us study the homotopy type of the $W'_j$. 
Let us call $(p_i,t_j)$ the vertex points with time projection $t_{j}$.
There is a map $\eta$ to the space $P^{\Moore}(\Conf(b_{j}-nm_{j},I^N\backslash \sqcup_i \{p_i\}))$; it sends $(\epsilon, \gamma)$ to just $\gamma$.
There is a map $\zeta$ in the other direction which sends a path 
\begin{align*}
    \gamma: [0,r]&\rightarrow \Conf(b_{j}-nm_{j},I^N\backslash \sqcup_i \{p_i\})
\end{align*}
to the pair $(\epsilon, \gamma')$
where $\epsilon$ is the minimum of $\frac r2$ and the distance between $p_j$ and the $\gamma(t)$ for all $t$. 
The path $\gamma'$ is $\gamma$ linearly rescaled to be defined on $[-\epsilon,\epsilon]$.
There is a homotopy between $\eta\circ \zeta$ and the identity given by continuously rescaling the path to $[-\epsilon,\epsilon]$, where $\epsilon$ depends continuously on $\gamma$.
There is a homotopy between $\zeta\circ\eta$ and the identity given by continuously shrinking $\epsilon$.
The space $W'_j$ is therefore equivalent to  $P^{\Moore}(\Conf(b_{j}-nm_{j},I^N\backslash \sqcup_i \{p_i\}))$; so by lemma \ref{connectedness}, it is $(N-2)$-connected.

It can be similarly shown that the spaces $W_j$ are $(N-2)$-connected.
This leads us to conclude that the space $\G_N(T)$ is $(N-3)$-connected.\qedhere

\end{proof}

\begin{cor}\label{treetree}
    For every finite $n$-ary rooted forest $T$, the space $\E_\infty(T)$ is contractible.
\end{cor}

\begin{proof}
    Lemma \ref{smallcubes} establishes that the fiber of the map $\E_N(T)\rightarrow \V_N(T)$, denoted by $\F_N(T)$, is equivalent to $\G_N(T)$.
    Since by lemma \ref{Gconnected}, $\G_N(T)$ is $(N-3)$-connected, the microfibration $\E_\infty(T)\rightarrow \V_\infty(T)$ is a weak equivalence and $\E_\infty(T)$ is contractible
\end{proof}

\subsection{Topological model}
We would now like to introduce a topological category and show that it is equivalent to the PROP associated to the cube cutting operad.
\begin{defi}
Let $\HT_N$ be the following non-unital category:\\
The space of objects is $\displaystyle \bigsqcup \UConf(n,I^N)\times\mathbb{R}$.\\
The space of morphisms is given by a tuple $(t_x,t_y,T,\phi,\epsilon)$ where $T$ is a $k$-ary forest, $(\phi,\epsilon)$ is an element of $\E_N(T)$ such that $\tilde{\phi}$ sends leaves to $t_x$ and roots to $t_y$, and $t_x<t_y$.
The source and target maps are given by $(\phi(t_x))$ and $(\phi(t_y))$.
\end{defi}
This category is not unital because there cannot be a morphism from $(x,t_x)$ to $(x,t_x)$.
When $N=\infty$, we set $\HT = \HT_\infty$.
The category $\C_N$ has an $E_N$-structure which induces an $E_N$-algebra structure on its classifying space.

There is a functor
\begin{align*}
F: \HT &\rightarrow \s(\op_{1,n})
\end{align*}
which sends each configuration to the number of points and each morphism in $\HT$ to the underlying tree. 
This functor induces a map on (thick) classifying spaces 
\begin{align*}
\mathbb{B}F: \mathbb{B}\HT &\rightarrow \mathbb{B}\s(\op_{1,n}).
\end{align*}
For a discrete category, the classifying space and thick classifying space are equivalent.
\begin{prop}\label{homeq}
    The map $\mathbb{B}F: \mathbb{B}\HT \rightarrow \mathbb{B}\s(\op_{1,n})$ is a homotopy equivalence of $E_\infty$-algebras.
\end{prop}

\begin{proof}
Let us consider the quotient map $f: \Conf(k,\mathbb{R}^\infty)\rightarrow \UConf(k,\mathbb{R}^\infty)$.
We can construct a category $f^*\HT$ using definition \ref{basechangeobjects}.
In order to apply proposition \ref{basechangeproperty}, we need to prove that $\HT$ is fibrant, ie that $(d_0,d_1):\Mor(\HT)\rightarrow \Ob(\HT)\times \Ob(\HT)$ is a fibration.
Let us split up $\Mor(\HT)$ and $\Ob(\HT)\times\Ob(\HT)$ into connected components.
The map $(d_0,d_1)$ then restricts to a collection of maps from the connected component of a given forest $T$ with $k$ leaves and $m$ roots, $\Mor_T(\HT)$, to $\UConf(k,\mathbb{R}^\infty)\times\UConf(m,\mathbb{R}^\infty)$.
Each of these maps is a microfibration - if we take a path of configurations and restrict the study to a small enough neighborhood of $(d_0,d_1)(\phi)$, 
we can continuously lift to a path of trees defined on $[0,\delta]$ for some small $\delta$ via affine paths as in figure \ref{redlift}, where the red paths represent liftings of the black sections of the trees.
\begin{figure}[h!] 
    \centering
    \def\svgwidth{0.8\hsize}
    \begingroup%
  \makeatletter%
  \providecommand\color[2][]{%
    \errmessage{(Inkscape) Color is used for the text in Inkscape, but the package 'color.sty' is not loaded}%
    \renewcommand\color[2][]{}%
  }%
  \providecommand\transparent[1]{%
    \errmessage{(Inkscape) Transparency is used (non-zero) for the text in Inkscape, but the package 'transparent.sty' is not loaded}%
    \renewcommand\transparent[1]{}%
  }%
  \providecommand\rotatebox[2]{#2}%
  \newcommand*\fsize{\dimexpr\f@size pt\relax}%
  \newcommand*\lineheight[1]{\fontsize{\fsize}{#1\fsize}\selectfont}%
  \ifx\svgwidth\undefined%
    \setlength{\unitlength}{425.15960934bp}%
    \ifx\svgscale\undefined%
      \relax%
    \else%
      \setlength{\unitlength}{\unitlength * \real{\svgscale}}%
    \fi%
  \else%
    \setlength{\unitlength}{\svgwidth}%
  \fi%
  \global\let\svgwidth\undefined%
  \global\let\svgscale\undefined%
  \makeatother%
  \begin{picture}(1,0.51709699)%
    \lineheight{1}%
    \setlength\tabcolsep{0pt}%
    \put(0,0){\includegraphics[width=\unitlength,page=1]{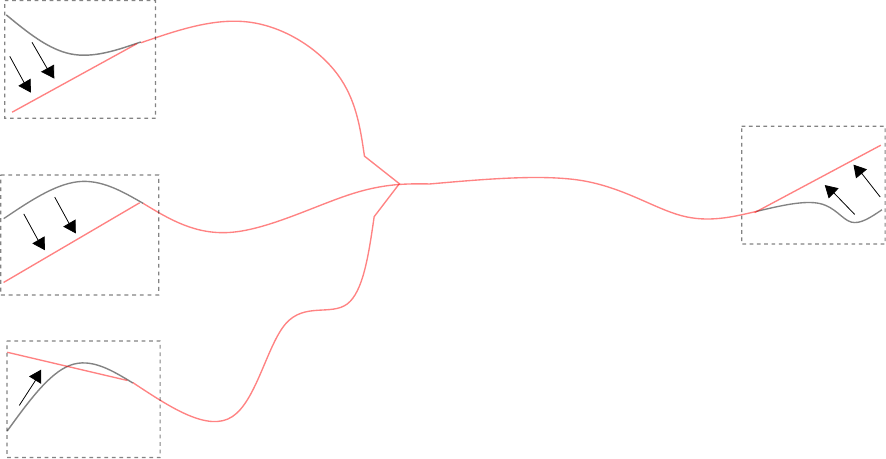}}%
  \end{picture}%
\endgroup%

    \caption{Microfibration} \label{redlift}
\end{figure}

The fiber above a point $(x,y)$ is the space of forests whose source and target are $x$ and $y$.
We can express it as the following fibered product:
\[P_x(\Conf(k,\mathbb{R}^\infty))\times_{\Conf(k,\mathbb{R}^\infty)} \E_\infty(T)\times_{\Conf(m,\mathbb{R}^\infty)}P^y(\Conf(m,\mathbb{R}^\infty)),\]
where $P_x$ and $P^y$ denote the Moore path spaces that respectively start at $x$ and end at $y$.
One way to see this is to think of a tree with fixed endpoints as a tree with no fixed endpoints, and then based paths of configurations to join leaves and roots to the fixed endpoints. 
The maps that induce the pullback $P_x(\Conf(k,\mathbb{R}^\infty))\rightarrow \Conf(k,\mathbb{R}^\infty)$ and $P^y(\Conf(m,\mathbb{R}^\infty))\rightarrow \Conf(m,\mathbb{R}^\infty)$ are target (resp. source) maps, so they are fibrations.
The pullback is therefore a homotopy pullback, and because $P_x\Conf(k,\mathbb{R}^\infty)$, $P^y(\Conf(m,\mathbb{R}^\infty))$, and $\E_\infty(T)$ are contractible, the homotopy pullback is contractible.

Each restriction $\Mor_T(\HT)\rightarrow\UConf(k,\mathbb{R}^\infty)\times\UConf(m,\mathbb{R}^\infty)$ is therefore a fibration, 
so $(d_0,d_1)$ is a fibration and $\HT$ is fibrant. Moreover, $f$ is 0-connected.
The category $\HT$ also has weak units: for any object $(x,t), x\in\UConf(n,\mathbb{R}^\infty)$, consider the constant path from $(x,t)$ to $(x,t+1)$.
This morphism induces an equivalence on morphism spaces into $(x,t)$ and $(x,t+1)$.
By proposition \ref{basechangeproperty}, there is an equivalence $\mathbb{B}\HT\simeq \mathbb{B}f^*\HT$.

Composing with the functor into $\s(\op_{1,n})$, we get a functor:
\begin{align*}
f^*\HT &\rightarrow \s(\op_{1,n}) &x\mapsto \vert x\vert, &&\phi\mapsto T_\phi.
\end{align*}
This functor is an equivalence on objects because the space of objects of $f^*\HT$, $\displaystyle \bigsqcup_k \Conf(k,\mathbb{R}^\infty)$, is equivalent to the space of objects of $\s(\op_{1,n})$, which is $\mathbb{N}$.
It is also an equivalence on morphisms because for a given $T$, the space of all $\phi$ such that $T_\phi = T$ is $\E_\infty(T)$, which is contractible by corollary \ref{treetree}.

At every level, the space $N_p(f^*\HT)$ is a pullback of the form $$\Mor(f^*\HT)\times_{\Ob(f^*\HT)}\dots \times_{\Ob(f^*\HT)}\Mor(f^*\HT),$$ which is also a homotopy pullback.
Therefore, $N_p(f^*\HT)\simeq N_p(\s(\op_{1,n}))$.
By lemma \ref{realizations}, this induces a weak equivalence on thick geometric realizations.
For the category $\s(\op_{1,n})$, the regular and thick realizations are equivalent.
Therefore, $\mathbb{B}F$ is a weak equivalence.
Moreover, it preserves the $E_\infty$-structure.
\end{proof}

\begin{rema}
    The groups associated to the category $\HT_2$ are related to the braided Higman--Thompson groups  \cite{brin,dehornoy,AC20, skipperwu}.
\end{rema}

\section{Delooping result}\label{deloopings}
In this section, we will set up what we need to zoom in on the classifying space of our category in order to find an expression of it as a loop space.
We would like to compare the space of embeddings of trees and a space of ``local'' images of trees via a scanning map.
We cannot directly show that this map is an equivalence, so we will zoom in on each dimension separately.
Zooming in on a chosen dimension means that we endow our space with a topology such that we are able to slide objects away to infinity along that direction.

\subsection{Semi-simplicial resolution of the classifying space of the category $\HT$}

\begin{defi}
Let $J$ be a subinterval of $\mathbb{R}$, $k,N$ two integers, $T$ a finite rooted $n$-ary forest.
Let $\F^k_N(T,J)$ be the space of embeddings of $T$ inside $J \times \mathbb{R}^k\times I^{N-k}$ defined in the same way as $\E_N(T)$ (\ref{embeddingspaces}) with one small modification:
every embedding is prolonged by constant paths so that its projection to $J$ is surjective.
The conditions (i) through (iv) in the definition of $\E_N(T)$ are also required.
When $J = \mathbb{R}$, we simply write $\F^k_N(T)$.
\end{defi}

\begin{defi}
    Let $k\geq 0$ and $M$ be a submanifold of $J\times\mathbb{R}^k\times I^{N-k}$ of the form $J'\times P'$, where $J'$ is a subinterval of $J$ and $P'$ is a submanifold of $\mathbb{R}^{k}\times I^{N-k}$. 
    Let 
    \[\Phi(J\times\mathbb{R}^k\times I^{N-k},M) := \bigg( \bigcup_T \displaystyle \F^k_{N}(T,J)\bigg)/ \sim,\]
    where $\phi\sim \psi$ if $\im(\phi)\cap M = \im(\psi)\cap M$.\\
    We will apply this to the case where $J' = J\cap [-\frac R2;\frac R2]$ and $P'$ is the cube $B_R^{k}(0)$ of center 0 and side $R$ in dimension $k$.
    For $0\leq k\leq N$, let us set $$\Phi_{k}^{N}(J) := \lim_{R\in\mathbb{N}} \Phi_k^N(J\times\mathbb{R}^{k}\times I^{N-k}, J'\times B_R^{k}(0) \times I^{N-k}).$$
\end{defi}
When $J=\mathbb{R}$, we omit it in the notation.
Note that for $k=0$, the space $\Phi_0^N$ identifies embeddings of trees whose images coincide along every compact subspace of the time dimension.

\begin{prop}\label{deloopinglevelone}
There is a weak equivalence of $E_N$-algebras: 
\[\Phi_0^N \simeq \mathbb{B}\HT_N.\]
\end{prop}

\begin{proof}
    Let us construct a semi-simplicial resolution $X^N_\bullet$ of $\Phi_0^N.$
    The space of $p$-simplices $X^N_p$ is given by the subspace of $\Phi_0^N\times N_p\mathbb{R}$ consisting of all $(\phi,t_0,\dots, t_p)$ where
    \begin{align*}
    \forall 0\leq i\leq p, \ t_i \text{ is at a distance larger than } \epsilon \text{ from a junction.}
    \end{align*}
    The face maps are given by:
    \begin{align*}
        d_k: X^N_p&\rightarrow X^N_{p-1}\\
        (\phi,t_0,\dots, t_p) &\mapsto (\phi,t_0,\dots, \hat{t_k},\dots, t_p).
    \end{align*}
  
    Let us denote by $\phi_{\vert[t_i,t_{i+1}]}$ the space $\im(\phi)\cap(I^N\times [t_i,t_{i+1}])$.
    Note that if $(\phi,\epsilon) \in \F^0_N(T)$, $t_i,t_{i+1} \in\mathbb{R}$, then there exists a $n$-ary subforest $T'$ of $T$ such that $(\phi_{\vert[t_i,t_{i+1}]},\epsilon) \in \E_N(T')$.
    There is a semi-simplicial map 
    \begin{align}\label{etap}
    \eta_p: X^N_p\rightarrow N_p\HT_N
    \end{align}
    which sends $(\phi,t_0,\dots, t_p)$ to $\left((t_0,t_1,\phi_{\vert[t_0,t_1]}),\dots,(t_{p-1},t_p,\phi_{\vert_{[t_{p-1},t_p]}})\right)$. 
    Informally, the map forgets what happens before $t_0$ and after $t_p$.
    
    The map $X^N_p\rightarrow N_{p}\mathbb{R}$ sending $(\phi, t_0,\dots, t_p)$ to $(t_0,\dots, t_p)$ is a fibration.
    The fiber over a point $(t_0,\dots, t_p)$ can be expressed as:
    \begin{align*}
    X'_p =\Phi_0^N((-\infty,t_0]\times I^N)&\times_{\UConf(n,I^N)}\dots\times_{\UConf(n,I^N)}\Phi_0^N([t_k,t_{k+1}]\times I^N)\\
    &\times_{\UConf(n,I^N)}\dots\times_{\UConf(n,I^N)}\Phi_0^N([t_p,\infty)\times I^N),
    \end{align*}
    because an embedding of a tree along $\mathbb{R}$ is obtained from a collection of embeddings into subintervals joined at their endpoints.
    The map $N_p\HT_N\rightarrow N_p\mathbb{R}$ is also a fibration. We then obtain the following diagram:
    \begin{center}
        \begin{tikzcd}
            X'_p \arrow{r}\arrow{d}&Y'_p\arrow{d}\\
            X_p\arrow{r}{\eta_p}\arrow{d} &N_p\HT_N\arrow{d}\\
            N_p\mathbb{R}\arrow{r}{=}&N_p\mathbb{R}.
        \end{tikzcd}
    \end{center}
    If we lift the map $\eta_p$ to the fibers above a point $(t_0,\dots, t_p)$, we get a map from $X'_p$ to the space
    \[Y'_p = \Phi_0^N([t_0,t_1])\times_{\UConf(n,I^N)}\dots\times_{\UConf(n,I^N)}\Phi_0^N([t_{p-1},t_{p}])\]
    which forgets the two outer terms of the product in the expression of $X'_p$. 
    We therefore need to prove that forgetting these terms does not change the homotopy type of the product.
    
    There are maps
    \begin{align*}
    p_0: \Phi_0^N((-\infty,t_0])&\rightarrow \UConf(n,I^N)\times\mathbb{R}\\
    p_p: \Phi_0^N([t_p,\infty))&\rightarrow \UConf(n,I^N)\times\mathbb{R}
    \end{align*}
    which are pullbacks of target (resp. source maps) and as such fibrations, so the pullback
    \[\Phi_0^N((-\infty,t_0])\times_{\UConf(n,I^N)} W \times_{\UConf(n,I^N)}\Phi_0^N([t_p,\infty))\]
    is a homotopy pullback.
    
    Let us show that the two outer terms in this homotopy pullback are equivalent to $\UConf(n,I^N)$
    and so removing them does not change the homotopy type.
    Let us consider the map $q_0:\UConf(n,I^N)\times\mathbb{R}\rightarrow \Phi_0^N((-\infty,t_0])$ induced by the maps $\UConf(n,I^N)\times\mathbb{R}\rightarrow \Phi((-\infty,t_0]\times I^N,[-R,R]\times I^N)$
    which send a configuration to the constant map.
    Clearly, $p_0\circ q_0 = \id$. The following is a homotopy between $q_0\circ p_0$ and $\id$:
    \begin{align*}
    H: \Phi_0^N((-\infty,t_0])\times I &\rightarrow \Phi_0^N((-\infty,t_0])\\
    (\phi,\epsilon,s) &\mapsto (\gamma_s,\epsilon)
    \end{align*}
    where for $s<1$, $\gamma_s$ is the embedding $\phi$ shifted by $\frac{-s}{1-s}$ along the time dimension, prolonged by a constant path. 
    Notice that this homotopy is only well-defined for $s\in[0,1)$; however, we can set $\gamma_1$ to be the constant path on the configuration of leaves, and this makes the homotopy well-defined and continuous. 
    See figure \ref{gammas} for an illustration of this homotopy.
    Therefore, the map $p_0$ is an equivalence.
    Similarly, the map $p_p$ is an equivalence.

    \begin{figure}[h!]
        \centering
            \def\svgwidth{0.8\hsize}
            \begingroup%
  \makeatletter%
  \providecommand\color[2][]{%
    \errmessage{(Inkscape) Color is used for the text in Inkscape, but the package 'color.sty' is not loaded}%
    \renewcommand\color[2][]{}%
  }%
  \providecommand\transparent[1]{%
    \errmessage{(Inkscape) Transparency is used (non-zero) for the text in Inkscape, but the package 'transparent.sty' is not loaded}%
    \renewcommand\transparent[1]{}%
  }%
  \providecommand\rotatebox[2]{#2}%
  \newcommand*\fsize{\dimexpr\f@size pt\relax}%
  \newcommand*\lineheight[1]{\fontsize{\fsize}{#1\fsize}\selectfont}%
  \ifx\svgwidth\undefined%
    \setlength{\unitlength}{429.62641678bp}%
    \ifx\svgscale\undefined%
      \relax%
    \else%
      \setlength{\unitlength}{\unitlength * \real{\svgscale}}%
    \fi%
  \else%
    \setlength{\unitlength}{\svgwidth}%
  \fi%
  \global\let\svgwidth\undefined%
  \global\let\svgscale\undefined%
  \makeatother%
  \begin{picture}(1,1.08344235)%
    \lineheight{1}%
    \setlength\tabcolsep{0pt}%
    \put(0,0){\includegraphics[width=\unitlength,page=1]{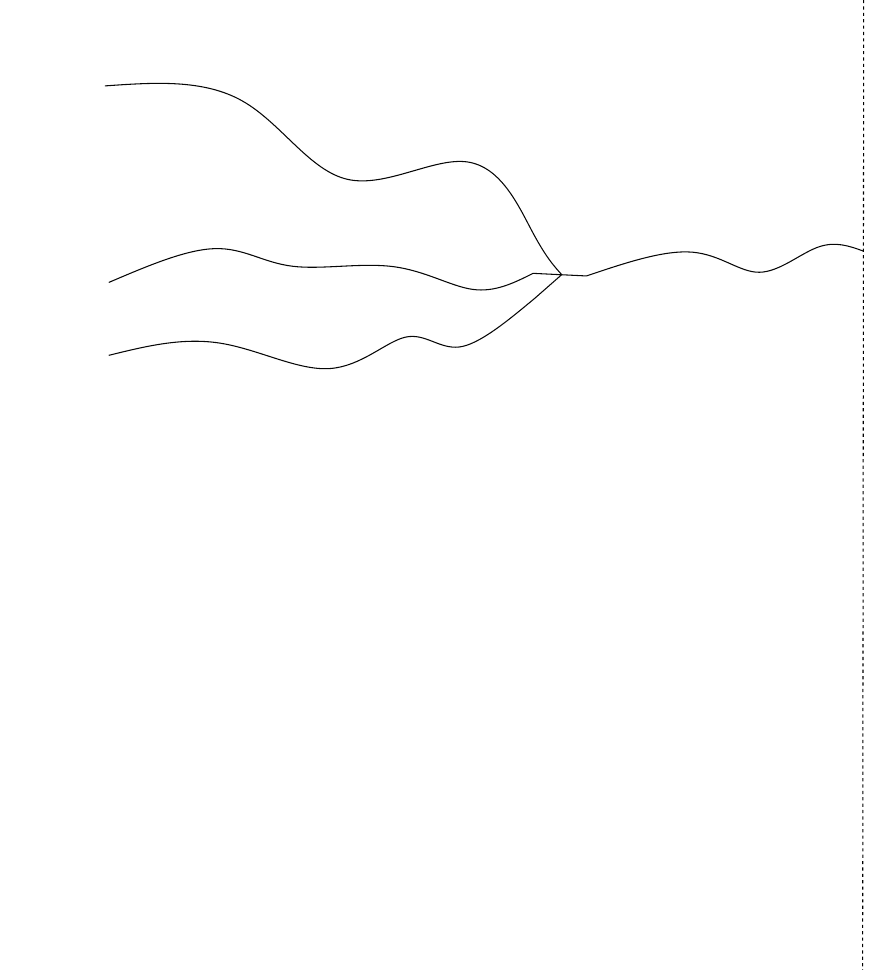}}%
    \put(0.91990227,1.05871889){\color[rgb]{0,0,0}\makebox(0,0)[lt]{\lineheight{1.25}\smash{\begin{tabular}[t]{l}$t_0$\end{tabular}}}}%
    \put(0.39635999,0.99327609){\color[rgb]{0,0,0}\makebox(0,0)[lt]{\lineheight{1.25}\smash{\begin{tabular}[t]{l}$\phi = \gamma_0$\end{tabular}}}}%
    \put(0,0){\includegraphics[width=\unitlength,page=2]{homotopie.pdf}}%
    \put(0.59123415,1.06162744){\color[rgb]{0,0,0}\makebox(0,0)[lt]{\lineheight{1.25}\smash{\begin{tabular}[t]{l}$t_0-\frac{s}{1-s}$\end{tabular}}}}%
    \put(0.3905429,0.47991381){\color[rgb]{0,0,0}\makebox(0,0)[lt]{\lineheight{1.25}\smash{\begin{tabular}[t]{l}$\gamma_s$\end{tabular}}}}%
    \put(0.3905429,0.17160569){\color[rgb]{0,0,0}\makebox(0,0)[lt]{\lineheight{1.25}\smash{\begin{tabular}[t]{l}$\gamma_1$\end{tabular}}}}%
    \put(0,0){\includegraphics[width=\unitlength,page=3]{homotopie.pdf}}%
    \put(-0.00211393,0.06544348){\color[rgb]{0,0,0}\makebox(0,0)[lt]{\lineheight{1.25}\smash{\begin{tabular}[t]{l}$s$\end{tabular}}}}%
  \end{picture}%
\endgroup%

            \caption{Illustration of the homotopy $H$ for $s=0$, $0<s<1$ and $s=1$.}\label{gammas}
        \end{figure}
    
    The map \eqref{etap} is an equivalence on fibers for every $p$, therefore it is a homotopy equivalence, and so the map $\eta$ induces an equivalence on semi-simplicial realizations:
    \[\vert\vert X^N_\bullet\vert\vert \simeq \vert\vert N_\bullet\HT_N\vert\vert.\]
        
    There is another map 
    \begin{align}
    \theta: \vert\vert X^N_\bullet\vert\vert\rightarrow \Phi_0^N 
    \end{align}
    which forgets the $(t_k)$. This map is a microfibration.
    Indeed, let us consider a square of the form
    \begin{center}
        \begin{tikzcd}[row sep = 0.4cm, column sep = 0.4cm]
        D^k\times\{0\}\arrow{dd}\arrow{dr} \arrow{rr}{A} &&\vert\vert X^N_\bullet\vert\vert\arrow{dd}{\theta}\\
        &D^k\times[0,\delta] \arrow[dashed]{ru}{H}\arrow{dl}\\
        D^k\times[0,1]\arrow{rr}{B} &&\Phi_0^N.
        \end{tikzcd}
    \end{center}
    We would like to find $\delta$ such that there exists a lift $H:D^k\times [0,\delta]\rightarrow \vert\vert X^N_\bullet\vert\vert$ in the above diagram (dashed arrow).
   Let $A:d\mapsto (x(d),\phi(d),t_0(d),\dots,t_p(d))$ and 
    $B:(d,s)\mapsto \psi(d,s)$.
    We want to construct $H: D^k\times[0;\delta]\rightarrow \vert\vert X_\bullet\vert\vert$ of the form $H(d,s) = (\chi(d,s),t_0(d,s),\dots,t_p(d,s))$ such that the $t_j(d,s)$ are far enough away from junctions for $\chi(d,s)$ (and $t_0(d,s)<\dots <t_p(d,s)$).
    We can take $\delta$ small enough so that for every $s\in[0,\delta]$, the $t_j(d)$ are still not closer than $\epsilon$ to a junction for $\psi(d,s)$, and then we can set
    $\chi(d,s)=\psi(d,s)$ and $t_j(d,s)= t_j(d)$.

   The fiber of $\theta$ over a point $(\phi,\epsilon)\in\Phi_0^N$ is the nerve of the topological poset $P\colonequals\{t\in\mathbb{R}, t \text{ is not within }\epsilon\text{ of a junction of }\phi.\}$.

   Let us consider the discrete version of this poset, $P^\delta$. There is an identity map $i: P^\delta\hookrightarrow P$.\\ 
    Let $D_{p,q} = N_{p+q+1}P$, viewed as a topological subspace of $N_pP\times N_qP^\delta$ via the map which takes a $(p+q+1)$-simplex in $NP$ and restricts it to its first $p$ coordinates on $NP$ and last $q$ coordinates on $NP^\delta$.
    
    This yields a semi-bisimplicial space $D$, with maps:
    \begin{align*}
        D_{p,q} \xrightarrow{\alpha_{p,q}} N_p P &&D_{p,q} \xrightarrow{\beta_{p,q}} N_q P^\delta.
    \end{align*}
    
     By proposition \ref{levelwisereal}, the map $\alpha_p: |\vert D_{p,\bullet}|\vert\rightarrow N_pP$ is a Serre microfibration. 
     The fiber of $\alpha_p$ over $(t_0,\dots, t_p) \in N_pP$ is the realization of the discrete simplicial space \[\{q\in\mathbb{N}, (s_0\leq\dots\leq s_q)\in P\backslash\{t_0,\dots, t_p\} \}.\]
    This is the nerve of a discrete poset which is $\mathbb{R}$ except for a finite union of closed intervals.  
     This poset is discrete and directed so its nerve is contractible. 
    Therefore the map $\alpha_p$ is a homotopy equivalence. 
    By lemma \ref{levelwisereal}, the map \[\vert\vert\alpha\vert\vert: ||D_{\bullet,\bullet}||\rightarrow ||N_\bullet P||\] is a homotopy equivalence. 
    But $\vert\vert i\vert\vert \circ ||\beta||\simeq ||\alpha||$ (\cite{grwI}, lemma 5.8), so $\vert\vert\alpha\vert\vert$ factors up to homotopy through the space $|\vert N_\bullet P^\delta|\vert$ (which is contractible since it is the realization of the nerve of a discrete directed poset). Thus, both spaces $||D_{\bullet,\bullet}||$ and $||N_\bullet P||$ are contractible.\\
    All that remains to show is that $P$ is a well-pointed category. This uses the topology of $P$ as a subposet of $\mathbb{R}$.
    We set 
    \begin{align*}
        u:\Mor(P)&\rightarrow I\times\Ob(P)\times\Ob(P) \\
        f_{x,y} &\mapsto (\min(1,y-x),x,y) \\
        f_{x,y} &\mapsto (1,x,y)
    \end{align*}
    and 
    \begin{align*}
        h:\Mor(P)\times I&\rightarrow \Mor(P)\\
        (f_{x,y},t)&\mapsto f_{x,tx+(1-t)y} \text{ if } [x,y] \subset P\\
        (f_{x,y},t)&\mapsto f_{x,y} \text{ else.}
    \end{align*}
    By proposition \ref{levelwisereal}, the thick geometric realization of $N_\bullet P$ is equivalent to its regular geometric realization. Therefore, $|N_\bullet P|$ is contractible. 
    
    It follows that the map $\theta$ is a homotopy equivalence, yielding the following zigzag of weak equivalences which also preserve the $E_N$-structure.
    \[\Phi_0^N\xleftarrow{\simeq}\vert\vert X^N_\bullet\vert\vert \xrightarrow{\simeq} \mathbb{B}\HT_N.\qedhere\]

\end{proof}

The inclusion $\mathbb{R}^N\hookrightarrow \mathbb{R}^{N+1}$ induces functors $\HT_N\rightarrow \HT_{N+1}$ and maps $\Phi_0^N\rightarrow \Phi_0^{N+1}$.
We would like these maps to be compatible with the zigzag exhibited above.
There are also induced maps $\vert\vert X^N_\bullet\vert\vert \rightarrow \vert\vert X^{N+1}_\bullet\vert\vert$.
\begin{lem}\label{compatibletowers}
    For every $N<\infty$, the following square commutes:
    \begin{center}
        \begin{tikzcd}
            \mathbb{B}\HT_N\arrow{d} &\vert\vert X^N_\bullet\vert\vert \arrow{l}\arrow{r}\arrow{d}&\Phi_0^N\arrow{d}\\
            \mathbb{B}\HT_{N+1} &\vert\vert X^{N+1}_\bullet\vert\vert \arrow{l}\arrow{r}&\Phi_0^{N+1}.
        \end{tikzcd}
        \end{center}
    \end{lem}
    
    \begin{proof}
    All the vertical maps are induced by the inclusion $\mathbb{R}^N\hookrightarrow \mathbb{R}^{N+1}$ so this follows.
    \end{proof}

\subsection{Delooping in higher dimensions}

We now carry out the zooming process in every dimension.
The space $\Phi_0^N$ is not path-connected, so we will need to consider its group completion. 
We therefore have the following proposition:
\begin{prop}\label{deloopingleveltwo}
    There is an equivalence $B \Phi_0^N \simeq \Phi_1^N$.
\end{prop}

\begin{proof}
We consider the monoid $\tilde{\Phi}_0^N$ whose elements are tuples $(\phi,a,b)$, with $a,b\in\mathbb{R}$ 
and $\phi$ a map from $\mathbb{R}$ to $\mathbb{R}\times [a,b]\times I^{N-1}$, 
which is the image by an affine map of an element in $\Phi_0^N$.
The monoid operation is obtained by stacking.
This monoid is equivalent to $\Phi_0^N$ as an $E_N$-algebra.

Let us consider the group completion $\Omega B \tilde{\Phi}_0^N$ for the stacking operation.
Let us show that $B \tilde{\Phi}_0^N \simeq \Phi_1^N$.

We use the same method as in the proof of proposition \ref{deloopinglevelone}, so we will go into less detail in this proof.
We define a semi-simplicial space $Z_\bullet$. 
The space of $p$-simplices $Z_p$ is the subspace of $\Phi_{1}^N\times N_p\mathbb{R}$ containing all $(\phi,t_0,\dots, t_p)$ 
such that for every $0\leq k \leq p$, $\phi$ does not intersect the hyperplanes $\mathbb{R}^{1}\times \{t_k\}\times I^{N-1}$.
The face maps are given by forgetting one of the $t_i$.

There is a semisimplicial map $\eta: Z_\bullet\rightarrow B_\bullet \tilde{\Phi}_0^N$:
on $p$-simplices, this map sends $(\phi,t_0,\dots, t_p)$ to the $p$-simplices of the nerve of the form 
$\left( (\phi_1, t_0,t_1),\dots, (\phi_p, t_{p-1}, t_p)\right)$,
where $\phi_i$ is the connected components of $\phi$ which are strictly contained in $\mathbb{R}\times [t_{i-1}, t_i]\times I^{N-1}$.
Basically, the map $\eta_p$ forgets what happens before $t_0$ and after $t_p$.

There is also an inclusion $i_p: B_p \tilde{\Phi}_0^N \rightarrow Z_p$.
The map $\eta_p\circ i_p$ is the identity, and the map $i_p\circ \eta_p$ is homotopic to the identity via the homotopy 
which slides away linearly the components of $\phi$ outside of $\mathbb{R}\times [t_{0}, t_p]\times I^{N-1}$ to the point at infinity of $Z_p$.
Therefore, the map $\eta_p$ is a homotopy equivalence for all $p$, and so the map $\eta$ induces an equivalence on thick realizations:
\[\vert\vert Z_\bullet\vert\vert \simeq  \vert\vert B_\bullet \tilde{\Phi}_0^N\vert\vert\]

The thick realization of the monoid is equivalent to the ordinary geometric realization, so:
\[\vert\vert Z_\bullet\vert\vert \simeq B\tilde{\Phi}_0^N.\]

Consider the projection map $\epsilon: \vert\vert Z_\bullet\vert\vert \rightarrow \Phi_1^N$,
which at each level sends $(\phi,t_0,\dots, t_p)$ to $\phi$. 
This map is a microfibration as in the proof of proposition \ref{deloopinglevelone}.

The fiber over a point $\phi\in\Phi_1^N$ is the realization of the nerve of the poset $P:= \{t\in\mathbb{R}\vert \phi\cap\left(\mathbb{R}^{1}\times \{t_k\}\times I^{N-1}\right) = \emptyset\}$,
which is contractible as in the proof of proposition \ref{deloopinglevelone}.
The microfibration $\epsilon$ has contractible fibers so it is a homotopy equivalence, yielding the following zigzag:
\[\Phi_1^N \xleftarrow{\simeq}\vert\vert Z_\bullet\vert\vert \xrightarrow{\simeq} B\tilde{\Phi}_0^N.\]
\end{proof}

\begin{prop}
For all $N$ and $1\leq k \leq N-1$, there is a weak equivalence:
\begin{align}\label{delooping}
\Phi_{k}^N \simeq \Omega \Phi_{k+1}^N,
\end{align}
\end{prop}


\begin{proof}
    We are going to use lemma \ref{semiSegal}.
    Let us construct a semi-simplicial space $X_\bullet$ such that $X_1\simeq \Phi^N_{k-1}$ and $\vert\vert X_\bullet\vert\vert \simeq \Phi_k^N$.
    
    Let $X_\bullet$ be the semi-simplicial space whose space of $p$-simplices is the subspace of $\Phi_k^N\times N_p\mathbb{R}$ of all $(\phi,t_0,\dots, t_p)$ where:
    \begin{align*} 
        &t_0<\dots <t_p\\
        &\phi \text{ does not intersect the space } \mathbb{R}^k\times \{t_k\}\times I^{N-k}, \text{ for every } 0\leq k \leq p.
    \end{align*}
    Face maps are given by forgetting one of the $t_i$.
    The Segal maps
    \begin{align*}
        X_n &\rightarrow X_1\times_{X_0}\dots\times_{X_0} X_1\\
        (\phi,t_0,\dots,t_n) &\mapsto ((\phi,t_0,t_1),\dots,(\phi,t_{n-1},t_n))
    \end{align*}
    are equivalences, so $X_\bullet$ is a semi-simplicial Segal space.

    We claim that $\Phi_{k-1}^N$ is a deformation retract of $X_1$.
    Let us define a map $l:X_1\rightarrow \Phi^N_{k-1}$ which restricts $(\phi,t_0,t_1)$ to $\phi\cap (\mathbb{R}^{k}\times (t_0,t_1) \times I^{N-k})$ and then rescales it linearly to an element of $\Phi^N_{k-1}$.
    We also consider the map $m:\Phi^N_{k-1}\rightarrow X_1$ which sends $\phi$ to $(\phi,0,1)$.
    There is a homotopy between $m\circ l$ and the identity of $X_1$. This homotopy is obtained by pushing off whatever is outside of $[t_0,t_1]$ to infinity and simultaneously rescaling $[t_0,t_1]$ to $[0,1]$.
    Moreover, $l\circ m = \id$.
    
    There is a map
    \begin{align}\label{fib}
        \epsilon: \vert\vert X_\bullet\vert\vert\rightarrow \Phi_k^N
    \end{align}
    which forgets the $(t_k)$. It is a microfibration as in the proof of proposition \ref{deloopinglevelone}.

    Let us show that the fibers of this map are weakly contractible. Given a point $\phi \in \Phi_k^N$, the fiber of the map $\epsilon$ over $\phi$ can be seen as the nerve of the poset 
    $$P \colonequals \{t\in \mathbb{R}\mid \phi \text{ does not intersect the space }\mathbb{R}^k\times \{t\}\times I^{N-k}\}$$ equipped with the usual ordering on $\mathbb{R}$.
    As in the proof of \ref{deloopinglevelone}, we show that the nerve of $P$ is contractible.
    The map \ref{fib} is a microfibration with contractible fibers, so it is a homotopy equivalence.
    By lemma \ref{semiSegal}, $\Phi(\mathbb{R}^{k}\times I^{N+1-k})\simeq \Omega\Phi_k^N$.
    \end{proof}

\begin{cor}\label{loopdeloop}
  There is a weak equivalence 
  \[\Omega B \Phi_0^N \simeq \Omega^N \Phi_N^N\]
\end{cor}

\section{Homology of the space of local images}\label{zoom}
We would like to give a more in-depth description of the space $\Phi_N^N$; roughly, it is a kind of compactification.
The point at infinity can be seen as a ``black hole'' which swallows everything which reaches it.
With this in mind, let us introduce the following definitions.

\begin{defi}
Let $U_1$ be the subset of $\Phi_N^N$ which contains those $(\phi,\epsilon)$ such that there exists $r>0$ that satisfies the following conditions: 
$\im(\phi)\cap B(0,r)$ contains no internal vertices and at most one path.

Let $U_n$ be the subset of $\Phi_N^N$ which contains those $(\phi,\epsilon)$ such that there exists $r>0$ that satisfies the following conditions: 
$\im(\phi)\cap B(0,r)$ is an affine tree which has exactly one internal vertex.
\end{defi}

We denote by $U_{1n}$ the intersection $U_n\cap U_1$.

\begin{lem}\label{openpushout}
    The following square is homotopy cocartesian.
    \begin{center}
    \begin{tikzcd}
    U_{1n}\arrow{r}\arrow{d} &U_1\arrow{d}\\
    U_n\arrow{r} &\Phi_N^N.
    \end{tikzcd}
    \end{center}
\end{lem}

\begin{proof}
Because $U_1$ and $U_n$ are open sets, it suffices to prove that $\Phi_N^N = U_1\cup U_n$.
For every $\phi\in\Phi_N^N$, set $x$ a point of $\im(\phi)$ that minimizes the distance to 0. 
Denote by $r$ this minimal distance.
If $r>0$, then $B(0,\frac{r}{2})\cap\im(\phi)$ is empty, and $\phi\in U_1$.
If $r=0$, then if $x$ is the image of an internal vertex, there exists, by definition of $\Phi_N^N$, $\epsilon>0$ such that $\phi$ coincides on a cube of side $2\epsilon$ with the embedding of a standard junction, so $\phi\in U_n$.
If $r=0$ but $x$ is not the image of an internal vertex, then there exists $r'$ such that $\im(\phi)\cap B(0,r')$ contains a unique strand so $\phi\in U_1$.
Therefore, $\Phi_N^N$ is the union of $U_1$ and $U_n$, so it is equivalent to their homotopy pushout over their intersection $U_{1n}$.
\end{proof}

\begin{defi}
Let $U'_1$ be the space of triples $(\phi, \epsilon, r)$ where $(\phi,\epsilon) \in U_1$ and $r$ is such that $\im(\phi)\cap B(0,r)$ has no internal vertices and at most one path closest to the origin.\\
Let $U'_n$ be the space of triples $(\phi, \epsilon, r)$ where $\im(\phi)\cap B(0,r)$ is an affine tree which has exactly one internal vertex.
\end{defi}

\begin{lem}\label{projectionprime}
The projection maps $U'_1\rightarrow U_1$ and $U'_n\rightarrow U_n$ are weak equivalences.
\end{lem}

\begin{proof}
Let us prove that the map $U'_1\rightarrow U_1$ is a microfibration. Consider given the outer square:
\begin{center}
    \begin{tikzcd}[column sep = 0.4cm, row sep = 0.4cm]
    D^k\times\{0\}\arrow{dd}\arrow{dr} \arrow{rr}{A} &&U'_1\arrow{dd}{p}\\
    &D^k\times[0,\delta] \arrow[dashed]{ru}{H}\arrow{dl}\\
    D^k\times[0;1]\arrow{rr}{B} &&U_1.
    \end{tikzcd}
\end{center}
Let $A(d) = (\phi(d),\epsilon(d),r(d)) \in U'_1$ and $B(d,s) = (\phi'(d,s),\epsilon'(d,s))$.
Then, there are two options.
On the one hand, if $\phi(d)\cap B(0,r(d))=\emptyset$, then there exists $\delta$ and $r'(d,s)$ defined on $D^k\times[0,\delta]$ such that $\phi'(d,s)\cap B(0,r'(d,s))=\emptyset$.
Then $H(d,s) = (\phi'(d,s),\epsilon'(d,s),r'(d,s))$ is a lift.
On the other hand, if $\phi(d)\cap B(0,r(d))\neq \emptyset$, there is a single strand which intersects $B(0,r(d))$.
There exists $\delta>0$ such that for every $s\in[0,\delta]$, there is still only one strand which intersects $B(0,r'(s))$. for some $r'$.
This gives us a continuous map $r'$ defined on $D^k\times[0,\delta]$ which provides a lift in the diagram.\\
The fiber over $(\phi,\epsilon)$ contains all $r$ such that $(\phi,\epsilon,r)\in U'_1$. 
If $r$ and $r'$ are in the fiber, consider $r<r''<r'$.
By definition, $\phi\cap B(0,r)$ has at most one path with no internal vertices, and $\phi\cap B(0,r')$ as well, so $\phi\cap B(0,r'')$ also does.\\
Therefore, the fiber of $U'_1\rightarrow U_1$ is an interval, so it is contractible and the map $U'_1\rightarrow U_1$ is a weak equivalence.
Things proceed in almost the same fashion for the map $U'_n\rightarrow U_n$.
\end{proof}

\begin{defi}
Let $U''_{1n}$ be the space of all $(\phi, \epsilon, r, r'),$ where $(\phi, \epsilon,r)\in U'_n$ and $(\phi, \epsilon,r')\in U'_1$.
\end{defi}

\begin{rema}
    If $(\phi,\epsilon,r,r')\in U''_{1n}$, then necessarily $r>r'$, and the internal vertex is in $B(0,r)\backslash B(0,r')$.
\end{rema}

\begin{lem}
    The map $U''_{1n}\rightarrow U_{1n}$ is a weak equivalence.
\end{lem}

\begin{proof}
The projection $U''_{1n}\rightarrow U_{1n}$ is also a microfibration. The proof is similar to that of Lemma \ref{projectionprime}. The fiber is the space of pairs $(r,r')$ such that $(\phi, \epsilon, r)$ is an element of $U'_n$ and $(\phi, \epsilon, r')$ is an element of $U'_1$, so it is a product of intervals of $\mathbb{R}$ and therefore contractible.
\end{proof}

\begin{cor}
The following square is homotopy cocartesian.
    \begin{center}
    \begin{tikzcd}
    U''_{1n}\arrow{r}\arrow{d}&U'_1\arrow{d}\\
    U'_n\arrow{r} &\Phi_N^N
    \end{tikzcd}
    \end{center}
\end{cor}

\begin{proof}
We replace each $U_i$ in the expression of the pushout in lemma \ref{openpushout} with equivalent spaces. 
\end{proof}

\begin{lem}\label{Uprimen}
The space $U'_{n}$ is contractible.
\end{lem}

\begin{proof}
    Given an element $(\phi,\epsilon,r)\in U'_n$, the ball $B(0,r)$ contains a unique internal vertex at the point $v$. 
    We define a map $f: U'_n\rightarrow D^{N+1}=B(0,1)$ which sends $(\phi, \epsilon, r)$ to $\frac{v}{r}\in D^{N+1}$.
Let us also define a map $g: D^{N+1}\rightarrow U_n$ which takes a point in $D^{N+1}$ and sends it to the affine tree which has one internal vertex centered around that point, setting $\epsilon = r = 1$.
These maps are homotopy inverses: it is clear that $f\circ g$ is the identity. Moreover, a homotopy from $g\circ f$ to the identity of $U'_n$ is given by pushing to infinity everything outside the ball of radius $r$ and rescaling the radius (and $\epsilon$) to 1.
The space $U'_n$ is therefore contractible.
\end{proof}

\begin{lem}\label{unbrin}
There is a pair of homotopy equivalences 
\begin{center}
\begin{tikzcd}
    f': U'_1 \arrow{r} &S^N : g'. \arrow[l, shift right]
\end{tikzcd}
\end{center}
\end{lem}

\begin{proof}
    We think of $S^N$ as $D^N$ with the boundary collapsed to a point. We can denote this by $D^N\cup_{S^{N-1}}*$.
Let us define a map 
\begin{align*}
f': U'_1 &\rightarrow D^N\cup_{S^{N-1}} *
\end{align*}
which sends $(\phi,\epsilon,r)$ to $x_0:= \tilde{\phi}^{-1}(0)\cap B(0,r)$, rescaled by a factor of $\frac{1}{r}$, if $\im(\phi)\cap B(0,r)$ is not empty.
If $\im(\phi)\cap B(0,r)=\emptyset$, then $f'$ maps $\phi$ to $*$.
We define another map 
\begin{align*}
g': D^N\cup_{S^{N-1}} * &\rightarrow U'_1\\
x &\mapsto (\phi:t\mapsto (x,t),1,1)\\
* &\mapsto (\emptyset,1,1)
\end{align*}
These maps are continuous. When a path goes to infinity, it enters the equivalence class of the basepoint of $\Phi_N^N$. 
Notice that $f'\circ g' = \id$. Moreover, there is a homotopy $h'$ between $\id$ and $g'\circ f'$ which linearly straightens the unique path inside the ball to a constant path.
Therefore, $U'_1$ and $S^N$ are homotopy equivalent.
\end{proof}

\begin{lem}\label{intersection}
There is a pair of homotopy equivalences:
\begin{center}
    \begin{tikzcd}
        f'': U''_{1n} \arrow{r} &S^N : g''. \arrow[l, shift right]
    \end{tikzcd}
    \end{center}
\end{lem}

\begin{proof}
An element $(\phi,\epsilon,r,r')$ of $U''_{1n}$ has an internal vertex at a distance at least $r'$ of the origin and at most $r$.
We can define a map $f'': U''_{1n}\rightarrow S^N$ which sends an element $(\phi, \epsilon, r, r')$ to the internal vertex, rescaled to get an element in $S^N$, as in figure \ref{sphere}.

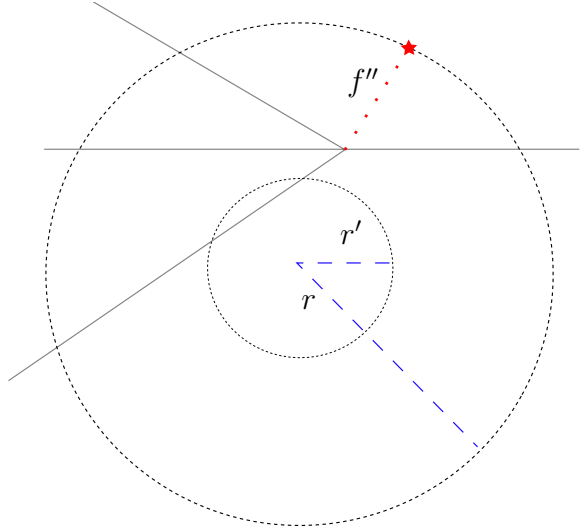
\begin{figure}[h!]
    \centering
        \def\svgwidth{0.6\hsize}
        \begingroup%
  \makeatletter%
  \providecommand\color[2][]{%
    \errmessage{(Inkscape) Color is used for the text in Inkscape, but the package 'color.sty' is not loaded}%
    \renewcommand\color[2][]{}%
  }%
  \providecommand\transparent[1]{%
    \errmessage{(Inkscape) Transparency is used (non-zero) for the text in Inkscape, but the package 'transparent.sty' is not loaded}%
    \renewcommand\transparent[1]{}%
  }%
  \providecommand\rotatebox[2]{#2}%
  \newcommand*\fsize{\dimexpr\f@size pt\relax}%
  \newcommand*\lineheight[1]{\fontsize{\fsize}{#1\fsize}\selectfont}%
  \ifx\svgwidth\undefined%
    \setlength{\unitlength}{368.87920229bp}%
    \ifx\svgscale\undefined%
      \relax%
    \else%
      \setlength{\unitlength}{\unitlength * \real{\svgscale}}%
    \fi%
  \else%
    \setlength{\unitlength}{\svgwidth}%
  \fi%
  \global\let\svgwidth\undefined%
  \global\let\svgscale\undefined%
  \makeatother%
  \begin{picture}(1,0.77422958)%
    \lineheight{1}%
    \setlength\tabcolsep{0pt}%
    \put(0,0){\includegraphics[width=\unitlength,page=1]{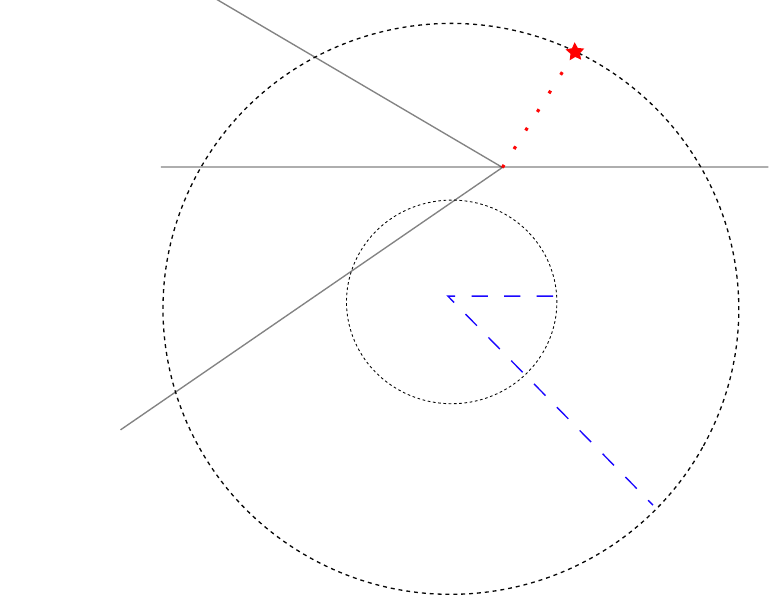}}%
    \put(0.64697744,0.41763987){\color[rgb]{0,0,0}\makebox(0,0)[lt]{\lineheight{1.25}\smash{\begin{tabular}[t]{l}$r'$\end{tabular}}}}%
    \put(0.5895434,0.32135343){\color[rgb]{0,0,0}\makebox(0,0)[lt]{\lineheight{1.25}\smash{\begin{tabular}[t]{l}$r$\end{tabular}}}}%
    \put(0.65711281,0.64246916){\color[rgb]{0,0,0}\makebox(0,0)[lt]{\lineheight{1.25}\smash{\begin{tabular}[t]{l}$f''$\end{tabular}}}}%
  \end{picture}%
\endgroup%

        \caption{The map $f''$ sends the tree to the point of the sphere labeled by a star.}\label{sphere}
    \end{figure}
Conversely, given a point $x$ in $S^N$, we can construct $\phi$ by placing an internal vertex inside $B(0,1)$ centered at $\frac{4x}{5}$. 
We can then note that $(\phi, 1, 1)$ is an element of $U'_{n}$ and $(\phi, 1, \frac{\pi}{20n})$ is an element of $U'_1$. 
Indeed, at most one path can intersect $B(0,\frac{\pi}{20n})$. 
This provides a map $g'': S^N\rightarrow U''_{1n}$ (taking $\epsilon = 1$).
The composition $f''\circ g''$ is the identity on the sphere. 
Let us build a homotopy $H''$ between $\id$ and $g''\circ f''$: it simply translates $\phi$ radially from its initial position centered around $x$ to a position centered at $\frac{4}{5}\frac{x}{\vert\vert x\vert\vert}$. 
The radius $r$ can be continuously rescaled to 1, and $r'$ is continuously rescaled to $\frac{\pi}{20n}$.
\end{proof}

Let $p$ be the map 
\[
S^N \xrightarrow{\pc} \vee_{n+1} S^N \xrightarrow{\psi_{n,1}} S^N 
\]
which first pinches then collapses the sphere $S^N$ as pictured in figure \ref{spheremap} into $n+1$ spheres: each point in a colored zone produces a ray which goes towards the center. 
The intersection point of the ray with the equatorial disk of radius $\frac{\pi}{20n}$ provides a map of the colored zone into $D^N$. 
All the unmarked zones of the outer sphere are sent to the basepoint of $S^N$. 
This provides a continuous pinch map $S^N \rightarrow \bigvee_{n+1} S^N$.
Then, we map each of the copies of the sphere into one sphere with the identity for n of the copies (the ones on the right) and the symmetry with respect to the hyperplane perpendicular to the vertical dimension of the tree for the last copy (the one on the left). 

\begin{lem}\label{specialmap}
    The following square commutes up to homotopy in both directions:
    \begin{center}
        \begin{tikzcd}
        U''_{1n} \arrow[d, bend right = 30, swap, "f''"]\arrow{r}{\pi} &U'_1\arrow[d, bend right = 30, swap, "f'"]\\
        S^N\arrow[u, bend right = 30, "g''"]\arrow{r}{p} &S^N\arrow[u, bend right = 30, "g'"],
        \end{tikzcd}
    \end{center}
    where the maps $f'$, $g'$, $f''$ and $g''$ are the homotopy equivalences defined in the proofs of lemmas \ref{intersection} and \ref{unbrin}.
\end{lem}

\begin{figure}
    \def\svgwidth{0.65\hsize}
    \begingroup%
  \makeatletter%
  \providecommand\color[2][]{%
    \errmessage{(Inkscape) Color is used for the text in Inkscape, but the package 'color.sty' is not loaded}%
    \renewcommand\color[2][]{}%
  }%
  \providecommand\transparent[1]{%
    \errmessage{(Inkscape) Transparency is used (non-zero) for the text in Inkscape, but the package 'transparent.sty' is not loaded}%
    \renewcommand\transparent[1]{}%
  }%
  \providecommand\rotatebox[2]{#2}%
  \newcommand*\fsize{\dimexpr\f@size pt\relax}%
  \newcommand*\lineheight[1]{\fontsize{\fsize}{#1\fsize}\selectfont}%
  \ifx\svgwidth\undefined%
    \setlength{\unitlength}{378.43890813bp}%
    \ifx\svgscale\undefined%
      \relax%
    \else%
      \setlength{\unitlength}{\unitlength * \real{\svgscale}}%
    \fi%
  \else%
    \setlength{\unitlength}{\svgwidth}%
  \fi%
  \global\let\svgwidth\undefined%
  \global\let\svgscale\undefined%
  \makeatother%
  \begin{picture}(1,1.13746448)%
    \lineheight{1}%
    \setlength\tabcolsep{0pt}%
    \put(0,0){\includegraphics[width=\unitlength,page=1]{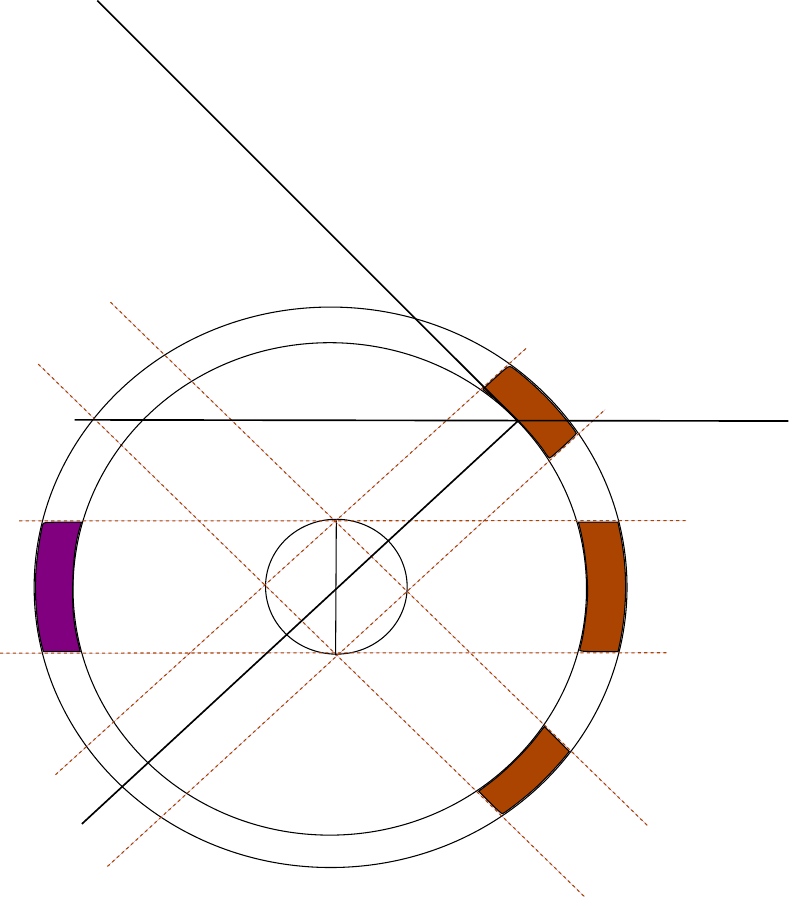}}%
  \end{picture}%
\endgroup%
    \caption{Illustration of the map $U''_{1n}\rightarrow U'_1$}
    \label{spheremap}
\end{figure}

\begin{proof}
    Let us first prove that $f'\circ\pi\simeq p\circ f''$.
Let $(\phi,\epsilon,r,r')\in U''_{1n}$. Then $f'(\pi(\phi,\epsilon,r,r'))$ is the point $\tilde{\phi}^{-1}(0)\cap B(0,r')$ rescaled to be an element of the sphere.
Going the other way, $f''(\phi,\epsilon,r,r')$ is the position of the internal vertex rescaled to be an element of the sphere. 
The map $p$ then sends that element to an element of the sphere which corresponds to where one ray intersects a smaller sphere, which is a rescaling of $\tilde{\phi}^{-1}(0)\cap B(0,r')$.

Let us look at the other square and show that $g'\circ p \simeq \pi\circ g''$. 
Let $x\in S^N$, then $g''(x) = (\phi_x, 1, 1, \frac{\pi}{20n})$ where $\phi_x$ is an affine tree with one junction at the point $\frac{4x}{5}$.
Then $\pi(g''(x))$ is $(\phi_x, 1, \frac{\pi}{20n})$. This is either empty, or the unique path going through the ball $B(0, \frac{\pi}{20n})$.

Going the other way, $\pc$ sends $x$ to the $i$th sphere in the wedge sum, where $i$ is also the strand which crosses the ball of radius $\frac{\pi}{20n}$.
Then $g'(p(x))$ is a constant path at $p(x)$. 
This constant path can be continously rectified to coincide with the path obtained by taking $\pi(g''(x))$.
Therefore, there is a homotopy between $\pi\circ g''$ and $g'\circ p$ (see figure \ref{carre2}).
\end{proof}

\begin{figure}[h!]
    \centering
        \def\svgwidth{0.65\hsize}
        \begingroup%
  \makeatletter%
  \providecommand\color[2][]{%
    \errmessage{(Inkscape) Color is used for the text in Inkscape, but the package 'color.sty' is not loaded}%
    \renewcommand\color[2][]{}%
  }%
  \providecommand\transparent[1]{%
    \errmessage{(Inkscape) Transparency is used (non-zero) for the text in Inkscape, but the package 'transparent.sty' is not loaded}%
    \renewcommand\transparent[1]{}%
  }%
  \providecommand\rotatebox[2]{#2}%
  \newcommand*\fsize{\dimexpr\f@size pt\relax}%
  \newcommand*\lineheight[1]{\fontsize{\fsize}{#1\fsize}\selectfont}%
  \ifx\svgwidth\undefined%
    \setlength{\unitlength}{499.84029935bp}%
    \ifx\svgscale\undefined%
      \relax%
    \else%
      \setlength{\unitlength}{\unitlength * \real{\svgscale}}%
    \fi%
  \else%
    \setlength{\unitlength}{\svgwidth}%
  \fi%
  \global\let\svgwidth\undefined%
  \global\let\svgscale\undefined%
  \makeatother%
  \begin{picture}(1,1.23994244)%
    \lineheight{1}%
    \setlength\tabcolsep{0pt}%
    \put(0,0){\includegraphics[width=\unitlength,page=1]{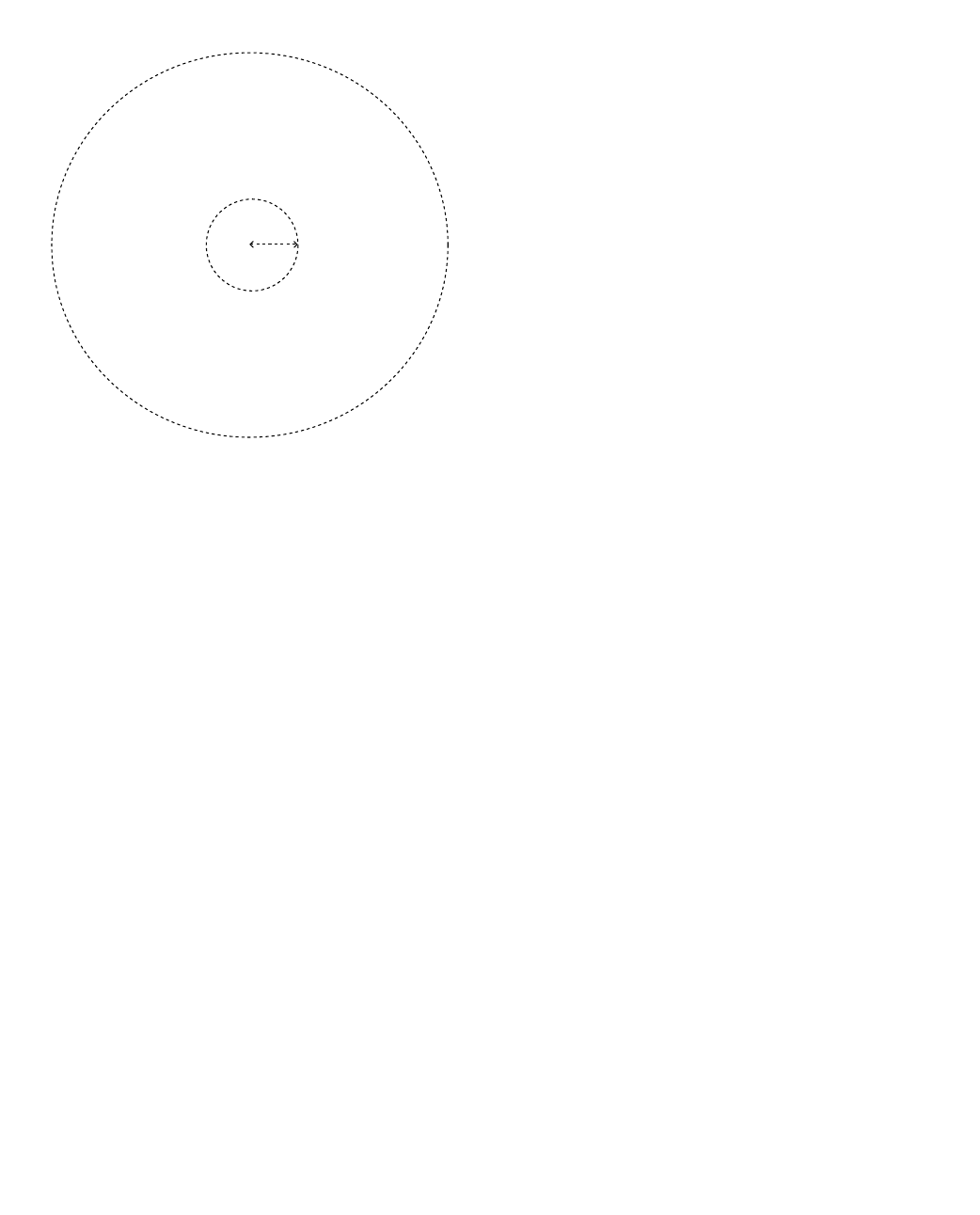}}%
    \put(0.24735996,0.96067361){\color[rgb]{0,0,0}\makebox(0,0)[lt]{\lineheight{1.25}\smash{\begin{tabular}[t]{l}$\frac{\pi}{10n}$\end{tabular}}}}%
    \put(0,0){\includegraphics[width=\unitlength,page=2]{carre2.pdf}}%
    \put(0.20108891,0.73169112){\color[rgb]{0,0,0}\makebox(0,0)[lt]{\lineheight{1.25}\smash{\begin{tabular}[t]{l}$g''(x)$\end{tabular}}}}%
    \put(0.72312148,0.74948769){\color[rgb]{0,0,0}\makebox(0,0)[lt]{\lineheight{1.25}\smash{\begin{tabular}[t]{l}$\pi(g''(x))$\end{tabular}}}}%
    \put(0.42057989,1.11728342){\color[rgb]{0,0,0}\makebox(0,0)[lt]{\lineheight{1.25}\smash{\begin{tabular}[t]{l}$x$\end{tabular}}}}%
    \put(0,0){\includegraphics[width=\unitlength,page=3]{carre2.pdf}}%
    \put(0.18203534,0.23974433){\color[rgb]{0,0,0}\makebox(0,0)[lt]{\lineheight{1.25}\smash{\begin{tabular}[t]{l}$\frac{\pi}{10n}$\end{tabular}}}}%
    \put(0,0){\includegraphics[width=\unitlength,page=4]{carre2.pdf}}%
    \put(0.15000157,0.02974488){\color[rgb]{0,0,0}\makebox(0,0)[lt]{\lineheight{1.25}\smash{\begin{tabular}[t]{l}$p(x)$\end{tabular}}}}%
    \put(0.36830607,0.39042202){\color[rgb]{0,0,0}\makebox(0,0)[lt]{\lineheight{1.25}\smash{\begin{tabular}[t]{l}$x$\end{tabular}}}}%
    \put(0,0){\includegraphics[width=\unitlength,page=5]{carre2.pdf}}%
    \put(0.16075,0.3176243){\color[rgb]{0,0,0}\makebox(0,0)[lt]{\lineheight{1.25}\smash{\begin{tabular}[t]{l}$p(x)$\end{tabular}}}}%
    \put(0,0){\includegraphics[width=\unitlength,page=6]{carre2.pdf}}%
    \put(0.73379954,0.03643855){\color[rgb]{0,0,0}\makebox(0,0)[lt]{\lineheight{1.25}\smash{\begin{tabular}[t]{l}$g'(p(x))$\end{tabular}}}}%
    \put(0,0){\includegraphics[width=\unitlength,page=7]{carre2.pdf}}%
  \end{picture}%
\endgroup%

        \caption{Commutativity of the second square}
        \label{carre2}
    \end{figure}

It turns out that each space $\Phi_N^N$ can be identified with a well-known space.
\begin{defi}
    The mod $(n-1)$ Moore spectrum $\mathbb{M}_{n-1}$ is the homotopy cofiber of $\mathbb{S}\xrightarrow{\times(n-1)}\mathbb{S}$. 
\end{defi}
The $N$th space of the Moore spectrum $\mathbb{M}^N_{n-1}$ is the homotopy cofiber of the map $S^N\xrightarrow{\times (n-1)}S^N$.

\begin{lem} \label{finaldelooping}
There is an equivalence of spaces between $\Phi_N^N$ and $\mathbb{M}^N_{n-1}$.
\end{lem}

\begin{proof}
    We deduce from lemmas \ref{Uprimen}, \ref{unbrin}, \ref{intersection} and \ref{specialmap} that the space $\Phi_N^N$ is obtained as the following homotopy pushout:
    \begin{center}
        \begin{tikzcd}
        &\bigvee_{n+1} S^N\arrow{dr}{\psi_{n,1}}\\
            S^N \arrow{rr}\arrow{d}\arrow{ur} &&S^N\\
            D^{N+1}.
        \end{tikzcd}
    \end{center}
    The homotopy colimit of this diagram is by definition the $N$-th space of the Moore spectrum.
    \end{proof}

We would like to extend this levelwise equivalence to an equivalence of spectra by showing compatibility with all the spectrum maps.
To that end, we need to build maps $\Phi_N^N\rightarrow \Omega\Phi_{N+1}^{N+1}$.
Recall that $\Phi_{N+1}^{N+1}$ is a limit of spaces $\Phi^{N+1}_{N+1,R}$, for $R\in\mathbb{R}.$
Let us define a map $\Phi_{N,R}^N\rightarrow \Omega\Phi_{N+1,R}^{N+1}$ which sends $(\phi,\epsilon)\in\Phi_{N,R}^N$ to the path
\begin{align*}
\gamma(t) = 
\begin{cases}
((\phi,t),\epsilon) \text{ for }-R< t<R, \\
* \text{ for } \vert t\vert \geq R.
\end{cases}
\end{align*}
These maps are well-defined for all $R$ and compatible, and therefore lift to a map $\omega_N: \Phi_N^N\rightarrow \Omega\Phi_{N+1}^{N+1}$.
There are maps $p_N: \mathbb{R}^N\times \Phi_0^N\rightarrow \Phi_N^N$ sending $(x,\phi,\epsilon)$ to $(\phi-x,\epsilon)$.
Because of the topology on $\Phi_N^N$, this factors through $S^N$ (if we view $S^N$ as the one-point compactification of $\mathbb{R}^N$) and yields a map $\Phi_0^N\rightarrow \Omega^N\Phi_N^N$.
We need to show that these maps are compatible with the maps $\Phi_0^N\rightarrow \Phi_0^{N+1}$ in the following sense:

\begin{lem}\label{compatibletwo}
    The following square commutes for all $N$:
\begin{center}
    \begin{tikzcd}
        \mathbb{R}^N\times \Phi_0^N\arrow{r}{p_N}\arrow{d} &\Phi_N^N\arrow{d}{\omega_N}\\
        \mathbb{R}^{N}\times \Phi_0^{N+1}\arrow{r}{p_{N+1}} &\Omega \Phi_{N+1}^{N+1}.
    \end{tikzcd}
\end{center} 
\end{lem}

\begin{proof}
Applying $p_N$ then $\omega_N$ to $(x,\phi,\epsilon)$ yields a path $\gamma$ in $\Omega \Phi_{N+1}^{N+1}$ with $\gamma(t) = ((\phi-x,t),\epsilon)$,
whereas taking the inclusion and then $p_{N+1}$ gives us a path with $\gamma'(t) = ((\phi,0) - (x,-t),\epsilon)$. Therefore, $\gamma = \gamma'$. 
\end{proof}

\begin{prop}
    There is a weak equivalence of spectra between the spectrum formed by the $(\Phi_N^N)$ with maps $\omega_N$, and the Moore spectrum.
\end{prop}

\begin{proof}
    Recall the spaces $U'_1(N)$, $U'_n(N)$, and $U''_{1n}(N)$ defined above. 
    We write $\Phi_N^{'N}$ for the pushout of $U'_1(N)\leftarrow U''_{1n}(N)\rightarrow U'_n(N)$.
    The $\Phi_N^{'N}$ form a spectrum, via the maps $\Phi_{N}^{'N}\rightarrow \Omega\Phi_{N+1}^{'N+1}$ which send $(\phi,\epsilon,r)$ to $(\omega_N(\phi,\epsilon),r)$. 

    Per lemma \ref{compatibletwo}, there exists a zigzag of spectra:
    \[\Phi_\bullet^\bullet \leftarrow \Phi_\bullet^{'\bullet} \rightarrow \mathbb{M}^\bullet.\]
    We need to show that the following diagram commutes:
    \begin{center}
    \begin{tikzcd}
        \Phi_N^N\arrow{d}{\omega_N} &\Phi_N^{'N}\arrow{l}\arrow{r}\arrow{d} &\mathbb{M}_{n-1}^N\arrow{d}\\
        \Omega\Phi_{N+1}^{N+1}&\Omega \Phi_{N+1}^{'N+1}\arrow{l}\arrow{r} &\Omega \mathbb{M}_{n-1}^{N+1}.
    \end{tikzcd}
    \end{center}

    Let us check in detail that both squares commute for $(\phi,\epsilon,r) \in U'_1$.
    The left-hand square commutes: going first vertically, $(\phi,\epsilon,r)$ is sent to $(\omega_N(\phi,\epsilon),r)$ then to $\omega_N(\phi,\epsilon)$.
    Going first horizontally, $(\phi,\epsilon,r)$ is sent to $(\phi,\epsilon)$ then $\omega_N(\phi,\epsilon)$.
    
    In the right-hand square, going first horizontally sends $(\phi,\epsilon,r)$ to the midpoint of $\im(\phi)\cap B(0,r)$ denoted by $x$: 
    depending on whether this intersection is empty or not, this is an element in the $N$th cell of $\mathbb{M}^N_{n-1}$ or its 0-cell. 
    The Moore spectrum map sends it to either an element in the $(N+1)$th cell of $\mathbb{M}^{N+1}_{n-1}$, or to the basepoint.
    Indeed, if $\phi$ has no internal vertex in $B(0,r)$, then neither does the translation of $\phi$.
    Similarly, $\phi\cap B(0,r)$ is empty if and only if the intersection of the translation of $\phi$ with a higher-dimensional disk is too. 
    
    In the other direction, $(\omega_N(\phi,\epsilon),r)$ is a path through $U'_1\subset \Phi_{N+1}^{'N+1}$ such that $(\omega_N(\phi,\epsilon),r)(t) = (\phi+(0,\dots,t),\epsilon,r)$.
    The bottom map then sends this to the path $t\mapsto (x,t)$. These two elements are equal, so the diagram commutes.
    We can similarly check commutativity of the diagrams for $(\phi,\epsilon,r) \in U'_n$.
    \end{proof}

\subsection{Group completion}\label{conclusion} 
We are now going to relate the colimit of the classifying spaces of the Higman--Thompson groups (see section \ref{trees}) denoted by $BV_{n,\infty}$ and the monoid given by their disjoint union $\bigsqcup_{1\leq r\leq n-1} BV_{n,r}$.

As stated in the proof of Proposition 1 in \cite{mcduffsegal}, the group completion of the monoid, denoted by $\Omega B(\bigsqcup_{1\leq r\leq n-1} BV_{n,r})$, is equivalent in homology to the colimit 
\[\operatorname{colim} \bigg(\bigsqcup_r BV_{n,r}\xrightarrow{\sigma} \bigsqcup_r BV_{n,r}\xrightarrow{\sigma}\dots\bigg),\]
where the map $\sigma: \bigsqcup BV_{n,r} \rightarrow \bigsqcup BV_{n,r}$ is induced by the maps $BV_{n,r}\rightarrow BV_{n,r+1}$ and the isomorphisms $V_{n,r}\iso V_{n,r+(n-1)}$.

Therefore:
\[\operatorname{colim} \bigg(\bigsqcup BV_{n,r}\xrightarrow{\sigma} \bigsqcup BV_{n,r}\xrightarrow{\sigma}\dots\bigg) \simeq \{1,\dots, n-1\}\times \operatorname{colim}\left(BV_{n,1}\rightarrow\dots\rightarrow BV_{n,r-1}\rightarrow BV_{n,1}\rightarrow \dots\right).\]
Every square in the following diagram commutes \cite[proof of Theorem~3.6]{szymikwahl}:
\begin{center}
    \begin{tikzcd}
        V_{n,1}\arrow{d}\arrow{r}&\dots\arrow{r}\arrow{d} &V_{n,r-1}\arrow{r}\arrow{d} &V_{n,1}\arrow{r}\arrow{d} &\dots\\
        V_{n,1}\arrow{r}&\dots\arrow{r} &V_{n,r-1}\arrow{r} &V_{n,r}\arrow{r} &\dots
    \end{tikzcd}
\end{center}
Therefore, we can see that:
\[\operatorname{colim}\left(BV_{n,1}\rightarrow\dots\rightarrow BV_{n,r-1}\rightarrow BV_{n,1}\rightarrow \dots\right) \simeq BV_{n,\infty}.\]
Finally, there is a homology equivalence:
\begin{align}\label{gpcompletion}
  \{1,\dots, n-1\}\times BV_{n,\infty} \overset{H_*}{\simeq} \Omega B(\bigsqcup BV_{n,r}).
\end{align}

\begin{teo}
There is a homology equivalence :
\[ BV_{n,\infty} \overset{H_*}{\simeq} \Omega^\infty_0 \mathbb{M}_{n-1}.\]
\end{teo}

\begin{proof}
Corollary \ref{loopdeloop} states that
\[\Omega B\Phi_0^N \simeq \Omega^N \Phi_N^N.\]
Combining this with corollary \ref{finaldelooping}, we get: 
\[\Omega B \Phi_0^N \simeq \Omega^N \mathbb{M}^N_{n-1}.\]
Now, proposition \ref{deloopinglevelone} gives us:
\[\Omega B (\mathbb{B}\HT_N)\simeq \Omega^N \mathbb{M}^N_{n-1}.\]
If we let $N$ go to infinity, using compatibility lemmas \ref{compatibletowers} and \ref{compatibletwo}, we get an equivalence with the infinite loop space:
\[\Omega B(\mathbb{B}\HT) \simeq \Omega^\infty \mathbb{M}_{n-1}.\]
Then using proposition \ref{homeq} and proposition \ref{propgroups}, we get:
\[\Omega B(\displaystyle \bigsqcup B V_{n,r}) \simeq \Omega^\infty \mathbb{M}_{n-1}. \]

Using equation \ref{gpcompletion}, we therefore get a homology equivalence with one component of the infinite loop space:
\[BV_{n,\infty} \overset{H_*}{\simeq} \Omega^\infty_0 \mathbb{M}_{n-1}.\qedhere\] 
\end{proof}
\printbibliography

\end{document}